\documentclass[a4paper, leqno, 12pt]{article}
\usepackage{amsmath,amsthm,stmaryrd}
\usepackage{amsfonts}
\usepackage{amssymb,latexsym}
\usepackage{enumerate}
\usepackage{color}
\usepackage{array}
\usepackage[all]{xy}
\usepackage{apptools}
\usepackage{mathrsfs}
\usepackage{tikz}

\usepackage{amsrefs}

\usepackage[left=2.5cm,right=2.5cm,top=3.0cm,bottom=3.0cm]{geometry}
\usepackage{hyperref}

\newtheorem{thm}{Theorem}[section]

\newtheorem{lm}[thm]{Lemma}
\newtheorem{prop}[thm]{Proposition}
\newtheorem{con}[thm]{Conjecture}

\theoremstyle{definition}
\newtheorem{df}[thm]{Definition}

\newtheorem{exa}[thm]{Example}

\numberwithin{equation}{section}

\def\NN{{\mathbb{N}}}
\def\ZZ{{\mathbb{Z}}}
\def\AA{{\mathbb{A}}}

\def\CA{{\cal A}}
\def\CC{{\cal C}}
\def\CB{{\cal B}}
\def\CD{{\cal D}}

\def\CF{{\cal F}}

\def\CK{{\cal K}}

\def\CL{{\cal L}}

\def\CR{{\cal R}}
\def\CS{{\cal S}}

\def\CX{{\cal X}}

 \def\wt{\widetilde} \def\und#1{\underline{#1}}

\def\epv {{ $\mbox{}$\hfill \qedsymbol }}
\def\wh#1{\widehat{#1}}

\def\Im{\textnormal{Im}}

\def\snull{\textnormal{ }}

\def\wt{\widetilde}
\def\ra{\rightarrow}

\def\lMOD{\mbox{{\rm -MOD}}}

\def\mod{\mbox{{\rm mod}}}

\def\ind{\mbox{{\rm ind}}}

\newcommand{\Rep}{\mathrm{Rep}}
\def\Mod{\mbox{{\rm Mod}}}
\def\Ext{\mbox{{\rm Ext}}}
\def\Coker{\mbox{{\rm Coker}}}  
\def\supp{\mbox{\rm supp}}

\def\MOD{{\rm MOD}} \def\Ind{{\rm Ind}} \def\ind{{\rm ind}}

\def\rad{\textnormal{rad}}

\def\Add{\textnormal{Add}}

\DeclareMathOperator{\ob}{ob}

\let\mod=\undefined
\DeclareMathOperator{\KG}{KG}
\DeclareMathOperator{\Ker}{Ker}
\DeclareMathOperator{\op}{op}
\DeclareMathOperator{\add}{add}
\DeclareMathOperator{\End}{End}
\DeclareMathOperator{\Hom}{Hom}
\DeclareMathOperator{\mod}{mod}

\begin{document}

\baselineskip=17pt


\title{On covering theory and its applications}

\author{Grzegorz Pastuszak${}^{*}$}

\date{}

\maketitle

\renewcommand{\thefootnote}{}
\footnote{${}^{*}$Faculty of Mathematics and Computer Science, Nicolaus Copernicus University, Chopina 12/18, 87-100 Toru\'n, Poland, e-mail: past@mat.umk.pl.}

\renewcommand{\thefootnote}{\arabic{footnote}}
\setcounter{footnote}{0}

\begin{abstract} The aim of this survey is to present applications of covering techniques in the theory of Krull-Gabriel dimension. We start with recalling fundamental facts of the classical covering theory of quivers and locally bounded categories. Then we present some recent results on covering theory of functor categories. These are interesting themselves, but also allow to relate Krull-Gabriel dimensions of locally bounded categories $R$ and $A$ when $R\ra A$ is a Galois covering functor. Finally, we concentrate on applications of our methods in describing Krull-Gabriel dimensions of various classes of algebras and locally bounded categories.
\end{abstract}

\tableofcontents

\section{Introduction} 

This is a survey article devoted to present fundamental concepts of covering theory and its applications in the representation theory of algebras. 

Throughout the paper, $K$ is a fixed algebraically closed field. We denote by $\MOD(K)$ and $\mod(K)$ the categories of all $K$-vector spaces and all finite dimensional $K$-vector spaces. An \textit{algebra} $A$ is always a finite dimensional associative basic $K$-algebra with a unit. By an \textit{$A$-module} we mean a right $A$-module. We denote by $\MOD(A)$ and $\mod(A)$ the categories of all $A$-modules and all finitely generated $A$-modules, respectively.

Covering theory emerged in the field of complex analysis in the study of Riemann surfaces. Nowadays it is a branch of topology, and particularly algebraic topology. In topology, coverings are special types of \emph{local homeomorphisms} which locally behave like projections. A standard example is the covering $p:\mathbb{R}\ra S^1$ of the unit circle $S^1$ by the space $\mathbb{R}$ of real numbers, defined by a well-known formula $p(x)=(\cos(2\pi x),\sin(2\pi x))$. It was first observed by C. Riedtmann in \cite{Rie} that these general ideas can be applied in Auslander-Reiten theory in the form of some special coverings of Auslander-Reiten quivers. Riedtmann's goal was what was most natural for the classical representation theory, namely, to describe the indecomposable modules over finite dimensional algebras. Her methods were soon organized into a full-fledged mathematical theory by P. Gabriel in \cite{Ga}, K. Bongartz and P. Gabriel in \cite{BoGa}, and R. Martinez-Villa and J. A. de la Pe\~na in \cite{MP}. Soon after, P. Dowbor and A. Skowro\'nski have introduced in \cite{DoSk,DoSk2} (see also \cite{DoLeSk}) further useful techniques in covering theory. Since then, covering theory of locally finite quivers and locally bounded $K$-categories (which are in a sense $K$-linear categorifications of quivers) has been successfully applied in the representation theory of finite dimensional algebras over a field. A class of coverings which is particularly useful in applications is the class of \emph{Galois coverings}. This paper concentrates on this special class.

Covering theory, as it appears in the representation theory of finite dimensional algebras, may be studied on several levels. The first one is that of quivers and locally bounded $K$-categories, see Section 2.2 and Definition \ref{dg} in particular. The second level is that of modules, since a covering functor between locally bounded $K$-categories naturally induces functors between module categories, see Section 2.3. The final level, at least from the classical perspective, is the level of functor categories.

Recall that functor categories have been widely used in the representation theory since the pioneer work by M. Auslander \cite{Au0}. In this influential paper the author shows that if $A$ is an algebra, then a contravariant $K$-linear functor $S:\mod(A)\ra\mod(K)$ is finitely presented (see Section 3.2) and simple if and only if there exists a right minimal almost split homomorphism $g:M\ra N$ in $\mod(A)$ \cite[IV]{AsSiSk} such that the sequence of functors $${}_{A}(-,M)\xrightarrow{{}_{A}(-,g)}{}_{A}(-,N)\ra S\ra 0$$ is a minimal projective presentation of $S$, see for example \cite[IV.6]{AsSiSk} for more details. Then in \cite[Corollary 3.14]{Au} Auslander proved that representation-finite algebras are exactly those algebras for which all finitely presented functors are of finite length. These two results are the starting point of \textit{functorial approach} to the representation theory of algebras which is now commonly used in the field. A natural continuation of the line of research started by Auslander is the study of the \textit{Krull-Gabriel filtration} of the category of finitely presented functors and then the \textit{Krull-Gabriel dimension}. We refer the reader to Section 5.1 for more details about these concepts and our motivation for exploring them.

The three levels we mention above have a common feature, based on the initial idea of \emph{local isomorphism} of the covering structure and the one which is covered. This is rather easy to observe in Definition \ref{dg} (see also the beginning of Section 2.2) and Theorems \ref{0t1}, \ref{t4}. Another common feature appears on the last two levels where we always have the \emph{pull-up functor}, see Sections 2.3 and 3.2. On the module level this functor is known to have the left and the right adjoints, called the \emph{push-down functors}. These functors are classically described in terms of direct sums and direct products of modules. However, on the level of functor categories, analogical descriptions of adjoints to the pull-up functor are in general no longer valid. Nevertheless, the pull-up still has both the left and the right adjoint functors which are constructed using the remarkable \emph{tensor product bifunctor} \cite{FPN,Mi} for categories of modules over small $K$-categories, see \cite[Theorem 3.6, Corollary 3.4]{P6'} and Theorem \ref{c1}. Since covering functors on the module level have similar descriptions as well \cite[Remark 3.3]{P6'}, we see the covering theory as the theory of the left and the right adjoint functors to the appropriate pull-up functor where the left adjoint is defined via the tensor product. This point of view is both unifying and fruitful.  

In this survey we aim to show fundamental concepts of the classical covering theory of locally bounded $K$-categories, covering theory of functor categories, and finally their applications in the theory of Krull-Gabriel dimension. The second section of the paper is devoted to covering theories of locally bounded $K$-categories and module categories. Sections 3 and 4 concern functor categories. The final Section 5 deals with applications to Krull-Gabriel dimension. 

The paper is largely based on our own results from \cite{P4,P6,J-PP1,P6'}. The latest covering techniques for functor categories which are described in Sections 3 and 4 come from \cite{P6'}. Although the paper is intended to be a review, we decided to include sketches of proofs of many theorems and facts. In this matter we focus on ideas, not technical details. We believe this will be useful to the reader.

\section{Covering theory of module categories}

This section recalls fundamentals of Galois covering theory of quivers, locally bounded $K$-categories and the associated functors between module categories. The notation introduced here is freely used in the paper. We mostly base on Section 2 of \cite{P6'} and Sections 2, 3 and 4 of \cite{P6}.

\subsection{Modules over locally bounded $K$-categories}

Assume that $R$ is a $K$-category and let $\ob(R)$ be the class of all objects of $R$. If $x,y\in\ob(R)$, then $R(x,y)$ denotes the space of all morphisms from $x$ to $y$. Following \cite{Ga,BoGa}, we say that a $K$-category $R$ is \textit{locally bounded} if and only if 

\begin{itemize}
	\item distinct objects of $R$ are not isomorphic,
	\item the algebra $R(x,x)$ is local, for any $x\in\ob(R)$,
	\item $\sum_{y\in\ob(R)}\dim_{K}R(x,y)<\infty$, $\sum_{y\in\ob(R)}\dim_{K}R(y,x)<\infty$, for any $x\in\ob(R)$.
\end{itemize}

The main examples of locally bounded $K$-categories are given by the \textit{bound quiver $K$-categories}. We recall necessary definitions below.

Assume that $Q=(Q_{0},Q_{1})$ is a \emph{quiver} where $Q_{0}$ is the set of vertices and $Q_{1}$ the set of arrows. Then $Q$ is \textit{finite} if both sets $Q_{0}$ and $Q_{1}$ are finite. The quiver $Q$ is \textit{locally finite} if the number of arrows in $Q_{1}$ starting or ending in any vertex is finite. If $\alpha\in Q_{1}$ is an arrow, then $s(\alpha)$ denotes its starting vertex and $t(\alpha)$ its terminal vertex. Assume that $x,y\in Q_{0}$. By a \textit{path} from vertex $x$ to vertex $y$ in $Q$ we mean a sequence $c_{1}\dots c_{n}$ in $Q_{1}$ such that $s(c_{1})=x$, $t(c_{n})=y$ and $t(c_{i})=s(c_{i+1})$, for $1\leq i<n$. This means that we compose arrows in the opposite direction to functions, consistently with \cite{AsSiSk}. We associate the \textit{stationary path} $e_{x}$ to each vertex $x\in Q_{0}$ and we set $s(e_{x})=t(e_{x})=x$.

The \textit{path $K$-category} $\und{KQ}$ of a locally finite quiver $Q$ is a $K$-category whose objects are the vertices of $Q$ and the $K$-linear space $\und{KQ}(x,y)$ of morphisms from $x$ to $y$ is generated by all paths from $y$ to $x$. The composition in $\und{KQ}$ is defined by concatenation of paths in $Q$, defined in a natural way. If $I$ is an \emph{admissible ideal} in $\und{KQ}$ \cite{Pog,BoGa}, then the pair $(Q,I)$ is called the \textit{bound quiver}. The associated quotient $K$-category $\und{KQ}\slash I$ is locally bounded and called the \textit{bound quiver $K$-category}. It is shown in \cite{BoGa} that any locally bounded $K$-category over an algebraically closed field $K$ is isomorphic with some bound quiver $K$-category.


Assume that $R$ is a locally bounded $K$-category. A \textit{right $R$-module} $M$ (or simply an \textit{$R$-module}) is a $K$-linear covariant functor $M:R^{\op}\ra\MOD(K)$, equivalently, a $K$-linear contravariant functor $M:R\ra\MOD(K)$. An $R$-module $M$ is \textit{finite dimensional} if and only if $\sum_{x\in\ob(R)}\dim_{K} M(x)<\infty$ and \emph{locally finite dimensional} if and only if $\dim_{K} M(x)<\infty$, for any $x\in\ob(R)$. Assume that $M,N:R^{\op}\ra\MOD(K)$ are $R$-modules. An \textit{$R$-module homomorphism} $f:M\ra N$ is a natural transformation of functors $(f_{x})_{x\in\ob(R)}$ where $f_{x}:M(x)\ra N(x)$ is a vector space homomorphism, for any $x\in\ob(R)$. The space of all homomorphisms from $M$ to $N$ is denoted by $\Hom_{R}(M,N)$. We usually write ${}_{R}(M,N)$ instead of $\Hom_{R}(M,N)$. Analogous notation for hom-spaces is used for any additive categories.

We denote by $\MOD(R)$, $\Mod(R)$ and $\mod(R)$ the categories of all $R$-modules, all locally finite dimensional $R$-modules and all finite dimensional $R$-modules, respectively. The full subcategories of $\Mod(R)$ and $\mod(R)$, formed by representatives of isomorphism classes of all indecomposable $R$-modules, are denoted by $\Ind(R)$ and $\ind(R)$, respectively. As usual, $\Gamma_{R}$ is the \emph{Auslander-Reiten quiver} of the category $\ind(R)$. Recall that the functor $$D=\Hom_{K}(-,K):\mod(R)\ra \mod(R^{\op})$$ is a duality between the category of right and left finite dimensional $R$-modules. 

Assume that $R=\und{KQ}\slash I$ is a bound quiver $K$-category. It is well known that the category $\MOD(R)$ of $R$-modules is equivalent with the category $\Rep_{K}(Q,I)$ of \emph{$K$-linear representations} of the bound quiver $(Q,I)$. Moreover, if $(Q,I)$ is finite, then there is an equivalence of $\Rep_{K}(Q,I)$ and the category of all modules over the \textit{bound quiver $K$-algebra} $KQ\slash I$, restricting to equivalence of categories of finite dimensional modules, see \cite[III 1.6]{AsSiSk} for details. More generally, any module $M\in\MOD(R)$ can be viewed as a direct sum $\widehat{M}=\bigoplus_{x\in\ob(R)}M(x)$ with a right action of $\widehat{R}=\bigoplus_{x,y\in\ob(R)}R(x,y)$ such that, for any $m\in\widehat{M}$ and $r\in\widehat{R}$, the element $m\cdot r$ is defined as in \cite[III 1.6]{AsSiSk}. In this language, $f:M\ra N$ is a homomorphism of $R$-modules, in the sense of a natural transformation of functors, if and only if $\widehat{f}:\widehat{M}\ra\widehat{N}$ is a homomorphism of $R$-modules in the usual sense, that is $\widehat{f}(m\cdot r)=\widehat{f}(m)\cdot r$. It is convenient in representation theory to use all three equivalent descriptions of module categories over locally bounded $K$-categories.    

Assume that $R$ is a locally bounded $K$-category and let $x\in\ob(R)$. Then $P_{x}=R(-,x)$ and $I_{x}=D(R(-,x))$ denote the indecomposable projective and the indecomposable injective (right) $R$-module, associated with the vertex $x$, respectively. Observe that these modules are finite dimensional, because $R$ is locally bounded. Moreover, they are indecomposable, essentially due to Yoneda Lemma and the definition of $R$, and all indecomposable projectives and injectives are of the form $P_x$ and $I_x$, respectively.

\subsection{Galois coverings}

Here we recall the notion of a Galois $G$-covering functor \cite{BoGa,MP}. For a general definition
of a covering functor we refer the reader to \cite{Ga,BoGa}. We start with Galois $G$-coverings of quivers and then we move on to locally bounded $K$-categories.

Assume that $Q'=(Q'_0,Q'_1)$ and $Q=(Q_0,Q_1)$ are quivers. Then a \emph{quiver morphism} $p:Q'\ra Q$ is a pair $(p_0,p_1)$ of functions $p_0:Q'_0\ra Q_0$ and $p_1:Q'_1\ra Q_1$ such that $s(p_{1}(\alpha))=p_0(s(\alpha))$ and $t(p_{1}(\alpha))=p_0(t(\alpha))$, for any arrow $\alpha\in Q'_1$. A quiver morphism $p:Q'\ra Q$ is an \emph{isomorphism} if and only if both functions $p_0$ and $p_1$ are bijections. An isomorphism $p:Q\ra Q$ is an \emph{automorphism}.

Assume that $p:Q'\ra Q$ is a morphism of quivers and $G$ is a group of automorphisms of the quiver $Q'$ acting freely on vertices of $Q'$. This means that $gx=x$ if and only if $g=1$, for any $g\in G$ and $x\in Q'_0$. Then $p:Q'\ra Q$ is a \emph{Galois $G$-covering of quivers} (or simply a \emph{Galois covering}) if and only if: 
\begin{itemize}
	\item the function $p_1:Q'_1\ra Q_1$ induces bijections $$\{\alpha:gx\ra y\mid \alpha\in Q'_1,g\in G\}\longleftrightarrow\{\beta:p_0(x)\ra p_0(y)\mid \beta\in Q_1\}\longleftrightarrow$$$$\longleftrightarrow\{\alpha:x\ra gy\mid \alpha\in Q'_1,g\in G\},$$ for any $x,y\in Q'_0$,
	\item the function $p_0:Q'_0\ra Q_0$ is surjective,
	\item $pg=p$, for any $g\in G$,
	\item for any vertices $x,y\in Q'_0$ such that $p(x)=p(y)$ there is $g\in G$ such that $gx=y$. 
\end{itemize} A quiver morphism $p:Q'\ra Q$ satisfies the above conditions if and only if it induces a quiver isomorphism $Q\cong Q'\slash G$ where $Q'\slash G$ is the \emph{orbit quiver}, see \cite{MP} for details.

In this paper we do not consider covering functors in general. Let us just mention that a quiver morphism $p:Q'\ra Q$ is a \emph{covering} if and only if $p_1$ induces a bijection between arrows starting (or ending) in a vertex $x$ of $Q'$ and starting (or ending) in a vertex $p(x)$ of $Q$, for any $x\in Q'_0$. Therefore we might say that quivers $Q'$ and $Q$ are \emph{locally the same} in this case. Recall that this is a common feature for coverings of topological spaces.

In covering theory a key role is played by \emph{universal coverings} \cite{MP} which are some special Galois coverings. We recall the necessary definitions below. First we remind the definition of the \emph{first fundamental group}.

Let $Q=(Q_0,Q_1,s,t)$ be a quiver and $\alpha\in Q_{1}$. We denote by $\alpha^{-1}$ the \textit{formal inverse} of $\alpha$ and we set $s(\alpha^{-1})=t(\alpha)$, $t(\alpha^{-1})=s(\alpha)$ and $(\alpha^{-1})^{-1}=\alpha$. If the composition of arrows $\alpha\beta$ exists, then $(\alpha\beta)^{-1}=\beta^{-1}\alpha^{-1}$. The set of all formal inverses of arrows from $Q_{1}$ is denoted by $Q_{1}^{-1}$. The elements of $Q_{1}$ are called \textit{direct arrows} and those of $Q_{1}^{-1}$ are \textit{inverse arrows}. A \textit{walk} from $x$ to $y$ in $Q$ is a path from $x$ to $y$ in $Q\cup Q_{1}^{-1}$. We denote by $W_{Q}$ the set of all paths in $Q$ and by $W_{Q}(x,y)$ the set of all paths in $Q$ from $x$ to $y$. 

A \emph{homotopy relation} in $Q$ is the smallest equivalence relation $\sim$ on $W_{Q}$ such that:
\begin{itemize}
  \item $\alpha\alpha^{-1}\sim e_{s(\alpha)}$ and $\alpha^{-1}\alpha\sim t_{s(\alpha)}$, for any arrow $\alpha\in Q_1$,
  \item if $u,w,u',w'\in W_Q$ are such that $u\sim u'$, $w\sim w'$ and $t(u)=s(w)$, $t(u')=s(w')$, then we have $uw\sim u'w'$.
\end{itemize} The \emph{first fundamental group} $\Pi_{1}(Q,a)$ of the quiver $Q$ with the base vertex $a\in Q_0$ is defined as the quotient set $W_{Q}(a,a)\slash\sim$ endowed with the multiplication $[u]_{\sim}\cdot[w]_{\sim}=[uw]_{\sim}$, for any $u,w\in W_{Q}(a,a)$. It is easy to see that the multiplication is correct and defines a group structure on $W_{Q}(a,a)\slash\sim$ with $[e_{a}]_\sim$ as the neutral element. Observe that if $w$ is a path from $a$ to $b$ in $Q$, then the function $\Pi_{1}(Q,a)\ra\Pi_{1}(Q,b)$ such that $[u]_{\sim}\mapsto[w^{-1}uw]_\sim$ is a group isomorphism. Hence the first fundamental group is independent of choice of the base vertex. It is also convenient to note that if $T$ is a maximal subquiver of $Q$ which is a tree, then $\Pi_{1}(Q,a)$ is a free group freely generated by the elements $[u_{\alpha}\alpha v^{-1}_\alpha]_\sim$ where $\alpha\in Q_{1}\setminus T_{1}$ and $u_{\alpha}$ ($v_{\alpha}$, respectively) is a path in $T$ connecting $a$ with $s(\alpha)$ ($a$ with $t(\alpha)$, respectively).

Assume that $Q=(Q_0,Q_1,s,t)$ is a connected quiver and $a\in Q_0$. We define the quiver $\wt{Q}=(\wt{Q_0},\wt{Q_1},\wt{s},\wt{t})$ in the following way:
\begin{itemize}
  \item the set $\wt{Q_0}$ of vertices in $\wt{Q}$ is the set $W_{Q}(a,-)\slash\sim$ of homotopy classes of all walks starting in $a$,
  \item the set $\wt{Q_1}$ of arrows in $\wt{Q}$ is the set of elements of the form $\alpha_{[u]_{\sim}}$ where $u\in W_{Q}(a,s(\alpha))$ such that $\wt{s}(\alpha_{[u]_{\sim}})=[u]_\sim$ and $\wt{t}(\alpha_{[u]_{\sim}})=[u\alpha]_\sim$.
\end{itemize} The function $p:\wt{Q}\ra Q$ such that $p([u]_{\sim})=t(u)$ and $p(\alpha_{[u]_{\sim}})=\alpha$ is called the \emph{universal covering} of a quiver $Q$. Observe that the fundamental group $\Pi_{1}(Q,a)$ acts on $\wt{Q}$ in the following natural way: if $[u]_{\sim}\in\Pi_{1}(Q,a)$, then $$[u]_{\sim}\cdot[w]_\sim=[uw]_\sim\textnormal{ and }[u]_{\sim}\cdot\alpha_{[w]_\sim}=\alpha_{[uw]_\sim}.$$ It can be shown that $p:\wt{Q}\ra Q$ is a Galois covering with $\Pi_{1}(Q,a)$ as the covering group. Moreover, if $q:Q'\ra Q$ is a Galois $G$-covering, then there is a normal subgroup $H$ of  $\Pi_{1}(Q,a)$ such that $G\cong\Pi_{1}(Q,a)\slash H$, $Q'\cong\wt{Q}\slash H$ and the following diagram commutes $$\xymatrix{\wt{Q}\ar[rr]^{p}\ar[dd]_{\pi}&&Q\\\\Q'\cong\wt{Q}\slash H\ar[uurr]_{q}}$$ where $\pi:\wt{Q}\ra\wt{Q}\slash H$ is a Galois covering with the covering group $H$. This property justifies the terminology for the universal covering $p:\wt{Q}\ra Q$.

In the following example we show the construction of the universal Galois covering $p:\wt{Q}\ra Q$ of the Kronecker quiver $Q$ from scratch, that is, based on the general definition.

\begin{exa}\label{e33} Assume that $$\xymatrix{\\ Q:=}\xymatrix{a \ar@/^0.4pc/[dd]^{\beta} \ar@/_0.4pc/[dd]_{\alpha}\\\\b}$$ is the Kronecker quiver. Observe that $\Pi_{1}(Q,a)$ is freely generated by the walk $\alpha\beta^{-1}$ and thus $\Pi_{1}(Q,a)\cong\ZZ$. The quiver morphism $p:\wt{Q}\ra Q$ such that the quiver $\wt{Q}$ is an infinite quiver of the form
$$\xymatrix{&\ar[dl]_{\beta_{-1}}[\beta\alpha^{-1}]_\sim\ar[dr]^{\alpha_{-1}}&&\ar[dl]_{\beta_0}[e_a]_\sim\ar[dr]^{\alpha_0}&&\ar[dl]_{\beta_1}[\alpha\beta^{-1}]_\sim\ar[dr]^{\alpha_1}\\\hdots [\beta\alpha^{-1}\beta]_\sim && [\beta]_\sim && [\alpha]_\sim && [\alpha\beta^{-1}\alpha]_\sim\hdots}$$ where:

\begin{itemize}
 \item $\alpha_{-1}=\alpha_{[\beta\alpha^{-1}]_\sim}$ and $\beta_{-1}=\beta_{[\beta\alpha^{-1}]_\sim}$,
  \item $\alpha_{0}=\alpha_{[e_a]_\sim}$ and $\beta_{0}=\beta_{[e_a]_\sim}$, 
  \item $\alpha_{1}=\alpha_{[\alpha\beta^{-1}]_\sim}$ and $\beta_{1}=\beta_{[\alpha\beta^{-1}]_\sim}$, and so on,
  \item $p([(\alpha\beta^{-1})^{i}]_\sim)=t((\alpha\beta^{-1})^{i})=a$, for $i\in\ZZ$,
  \item $p([(\alpha\beta^{-1})^{i}\beta]_\sim)=t((\alpha\beta^{-1})^{i}\beta)=b$, for $i\in\ZZ$,
  \item $p(\alpha_i)=\alpha$, $p(\beta_i)=\beta$, for $i\in\ZZ$
\end{itemize} is the universal covering of the quiver $Q$, that is, the Galois $G$-covering with the covering group $G=\Pi_{1}(Q,a)\cong\ZZ$. It is clear that such description is rather inconvenient. We usually simplify the notation on $\wt{Q}$ as follows: $$\xymatrix{&\ar[dl]_{\beta} a\ar[dr]^{\alpha}&&\ar[dl]_{\beta}a\ar[dr]^{\alpha}&&\ar[dl]_{\beta}a\ar[dr]^{\alpha}\\\hdots b && b && b && b\hdots}$$ This is supposed to indicate all fibers of $p$, equivalently, all $\ZZ$-orbits of the action. \epv
\end{exa} 

The above notions and constructions can be generalized to the case of bound quivers. In particular, one defines the first fundamental group $\Pi_{1}(Q,I)$ of a bound quiver $(Q,I)$. If $p:(\wt{Q},\wt{I})\ra(Q,I)$ is the universal covering of a bound quiver $(Q,I)$, then $\wt{Q}$ is constructed as before and $\wt{I}$ is obtained from $I$ by \emph{lifting}, in a suitable way, the relations generating $I$. Then $p:(\wt{Q},\wt{I})\ra(Q,I)$ is also a Galois covering with the covering group $\Pi_{1}(\wt{Q},\wt{I})$, and any Galois covering of $(Q,I)$ appears in a similar way as for quivers. We refer to \cite{MP} for details about these constructions. The crucial thing is that Galois coverings of bound quivers naturally induce Galois coverings on the level of the associated locally bounded $K$-categories. This is the source of the following definition. 

\begin{df}\label{dg} Assume that $R,A$ are locally bounded $K$-categories, $F:R\ra A$ is a $K$-linear functor and $G$ a group of $K$-linear automorphisms of $R$ acting freely on the objects of $R$ (i.e. $gx=x$ if and only if $g=1$, for any $g\in G$ and $x\in\ob(R)$). Then $F:R\ra A$ is a \textit{Galois G-covering} (or simply a \emph{Galois covering}) if and only if 
\begin{itemize}
	\item the functor $F:R\ra A$ induces isomorphisms $$\bigoplus_{g\in G}R(gx,y)\cong A(F(x),F(y))\cong\bigoplus_{g\in G}R(x,gy)$$ of vector spaces, for any $x,y\in\ob(R)$,
	\item the functor $F:R\ra A$ induces a surjective function $\ob(R)\ra\ob(A)$,
	\item $Fg=F$, for any $g\in G$,
	\item for any $x,y\in\ob(R)$ such that $F(x)=F(y)$ there is $g\in G$ such that $gx=y$. 
\end{itemize}
\end{df} As for quivers, the functor $F:R\ra A$ satisfies the above conditions if and only if $F$ induces an isomorphism $A\cong R\slash G$ where $R\slash G$ is the \textit{orbit category} \cite{BoGa}. In this case we recall its definition. Namely, the objects of $R\slash G$ are $G$-orbits $Gx=\{gx\mid g\in G\}$, $x\in\ob(R)$. The set of morphisms $(R\slash G)(Gx,Gy)$ is defined as the set of fix points of the following action of $G$ on $\prod_{z\in Gx}\prod_{t\in Gy}R(z,t)$: $$g(_{t}f_{z}:z\ra t)_{z\in Gx,t\in Gy}=(g(_{g^{-1}t}f_{g^{-1}z}):z\ra t)_{z\in Gx,t\in Gy}.$$ In other words, $(_{t}f_{z}:z\ra t)_{z\in Gx,t\in Gy}$ is a morphism if and only if $g(_{t}f_{z})={}_{gt}f_{gz}$.

As mentioned above, if $p:(Q',I')\ra(Q,I)$ is a Galois $G$-covering of bound quivers, then the induced $K$-linear functor $F_{p}:\und{KQ'}\slash I'\ra \und{KQ}\slash I$ is a Galois $G$-covering in the above sense. Moreover, since the field $K$ is algebraically closed, any Galois covering in the sense of Definition \ref{dg} appears in this way. Nevertheless, in covering theory results are usually formulated in terms of locally bounded $K$-categories. In this paper we follow this tradition. 

The next example describes Galois coverings of \emph{trivial extensions} which are ubiquitous in representation theory, see for example \cite{ErKeSk} and \cite{Sk4}. Here we recall a general construction having two important special cases, namely, where the covering category is \emph{repetitive} or \emph{cluster repetitive}. We use these constructions in Section 5.

\begin{exa}\label{e44} Assume that $C$ is an algebra and $E$ is a non-zero $C$-$C$-bimodule. Consider a locally finite dimensional matrix $K$-algebra $C_{E}$ of the form
$$C_{E}=\left[\begin{array}{ccccc}\ddots&&&&0\\\ddots&C_{-1}&&&\\  &E_{0}&C_{0}&&\\&&E_{1}&C_{1}&\\
 0&& &\ddots&\ddots\end{array}\right]$$ where $C_{i}=C$ and $E_{i}=E$, for any $i\in\ZZ$. The multiplication is naturally induced from that of $C$ and the $C$-$C$-bimodule structure of $E$. The identity maps $C_{i}\ra C_{i-1}$ and $E_{i}\ra E_{i-1}$ induce an automorphism $\nu=\nu_{C_{E}}$ such that the orbit algebra $C_{E}\slash\langle\nu\rangle$ is isomorphic to the \textit{trivial extension} $C\ltimes E$ of $C$ by $E$. The algebra $C_{E}$ may be viewed as a locally bounded $K$-category as follows. Assume that $\{e_{1},\dots,e_{n}\}$ is a complete set of primitive
orthogonal idempotents of $C$. Then the objects of $C_{E}$ are of the form $e_{m,i}$, for $m\in\{1,\dots,n\}$, $i\in\ZZ$, and the morphism spaces are
defined as: $$C_{E}(e_{m,j},e_{l,i})=\left\{\begin{array}{cl}e_{l}Ce_{m}, & i=j,\\e_{l}Ee_{m}, & i=j+1,\\
0,& \textnormal{otherwise.}\end{array} \right.$$ Then the projection functor $C_{E}\ra C_{E}\slash\langle\nu\rangle\cong C\ltimes E$ is a Galois covering  with an  admissible torsion-free covering group $\langle\nu\rangle\cong\ZZ$.

In the case $E=D(C)$, the algebra $C_{E}$ is called the \textit{repetitive algebra} of $C$ (equivalently, the \textit{repetitive category}), denoted by
$\widehat{C}$,  and it is a self-injective algebra \cite{HW}. The automorphism $\nu_{\widehat{C}}$ is called the \textit{Nakayama automorphism} of
$\widehat{C}$ and $\widehat{C}\slash\langle\nu_{\widehat{C}}\rangle$ is isomorphic to the usual trivial extension algebra $T(C)= C\ltimes D(C)$. 

If $C$ is a tilted algebra and $E=\Ext_{C}^{2}(DC,C)$, then $C_{E}$ is called the \textit{cluster repetitive algebra} of $C$ and
denoted by $\check{C}$ \cite{ABS}. In this case, the trivial extension $C\ltimes E$ of $C$ by $E=\Ext_{C}^{2}(DC,C)$ is called the \textit{relation
extension algebra} and denoted by $\wt{C}$ \cite{ABS2}. Remarkably, it follows from \cite{ABS2} that $\wt{C}$ is a \textit{cluster-tilted algebra} in the
sense of \cite{BMR}. This means that $\wt{C}$ is the endomorphism algebra of a cluster-tilting object in a cluster category. Moreover, every cluster-tilted algebra occurs in that way. \epv
\end{exa}

\subsection{Covering functors}

A Galois $G$-covering functor $F:R\ra A\cong R\slash G$ induces three functors on the level of modules: $F_{\bullet}:\MOD(A)\ra\MOD(R)$ and $F_{\lambda},F_\rho:\MOD(R)\ra\MOD(A)$. These functors closely link the representation theories of $R$ and $A$. Here we aim to define these functors in detail and describe their main properties.  

Assume that $F\colon R\ra A$ is a Galois $G$-covering functor. Then the \textit{pull-up} functor $F_{\bullet}:\MOD(A)\ra\MOD(R)$ associated with $F$ is the functor $(-)\circ F^{\op}$. The pull-up functor is exact and has the left adjoint $F_{\lambda}:\MOD(R)\ra\MOD(A)$ and the right adjoint $F_{\rho}:\MOD(R)\ra\MOD(A)$ which are called the \textit{push-down} functors. We recall the description of the push-down functors below. First we introduce some terminology. 

Assume that $X_{1},\dots,X_{n}$ and $Y_{1},\dots,Y_{m}$ are objects of an abelian category $\CC$ and let $$f\in{}_{\CC}(\bigoplus_{i=1}^{n}X_{i},\bigoplus_{j=1}^{m}Y_{j})$$ be a morphism. Then $f=[f_{ji}]_{i=1,\dots,n}^{j=1,\dots,m}$ where $p_{j}:\bigoplus_{k=1}^{n}Y_{k}\ra Y_{j}$ is the split epimorphism, $u_{i}:X_{i}\ra\bigoplus_{k=1}^{m}X_{k}$ is the split monomorphism and $f_{ji}=p_{j}fu_{i}$, for any $i=1,\dots,n$, $j=1,\dots,m$. We say in this case that $f$ \textit{is defined by homomorphisms $f_{ji}$}. We assume similar terminology for morphisms between direct products of objects in $\CC$.

Assume that $M:R^{\op}\ra\MOD(K)$ is an $R$-module. We define the $A$-module $F_{\lambda}(M):A^{\op}\ra\MOD(K)$ in the following way. Assume that $a\in\ob(A)$ and $a=F(x)$, for some $x\in\ob(R)$. Then we have $$F_{\lambda}(M)(a)=\bigoplus_{g\in G}M(gx).$$ Assume that $\alpha\in A(b,a)$ and $a=F(x),b=F(y)$, for some $x,y\in\ob(R)$. Since $F$ induces an isomorphism $$\bigoplus_{g\in G}R(gy,x)\cong A(F(y),F(x)),$$ there are $\alpha_{g}:gy\ra x$, for $g\in G$, such that $\alpha=\sum_{g\in G}F(\alpha_{g})$. Then the homomorphism $$F_{\lambda}(M)(\alpha):F_{\lambda}(M)(a)\ra F_{\lambda}(M)(b)$$ is defined by homomorphisms $M(g\alpha_{g^{-1}h}):M(gx)\ra M(hy)$, for any $g,h\in G$. Assume that $f:M\ra N$ is an $R$-module homomorphism and $f=(f_{x})_{x\in\ob(R)}$, $f_{x}:M(x)\ra N(x)$. Then $F_{\lambda}(f):F_{\lambda}(M)\ra F_{\lambda}(N)$, $F_{\lambda}(f)=(\hat{f}_{a})_{a\in\ob(A)}$ and $\hat{f}_{a}:F_{\lambda}(M)(a)\ra F_{\lambda}(N)(a)$ is defined by homomorphisms $f_{gx}:M(gx)\ra N(gx)$, for any $g\in G$. For the module $F_{\rho}(M):A^{\op}\ra\MOD(K)$ we have $$F_{\rho}(M)(a)=\prod_{g\in G}M(gx)$$ and the rest of the definition is similar to the case of $F_{\lambda}$. We refer the reader to \cite[Remark 3.3]{P6'} for equivalent definitions of push-down functors in terms of the tensor product.

Observe that $F_{\lambda}$ is a subfunctor of $F_{\rho}$ and both functors coincide on the category of finite dimensional $A$-modules. It is important to note that general covering functors do not have this property. Moreover, if an $R$-module $M$ is finite dimensional, then $F_{\lambda}(M)$ is finite dimensional. Hence the functor $F_{\lambda}$ restricts to a functor $\mod(R)\ra\mod(A)$. This functor is also denoted by $F_{\lambda}$. 

Assume that $R$ is a locally bounded $K$-category, $G$ is a group of $K$-linear automorphisms of $R$ acting freely on the objects of $R$ and $g\in G$. Given $R$-module $M$ we denote by ${}^{g}M$ the module $M\circ g^{-1}$. Given $R$-module homomorphism $f:M\ra N$ we denote by ${}^{g}f$ the $R$-module homomorphism ${}^{g}M\ra {}^{g}N$ such that ${}^{g}f_{x}=f_{g^{-1}x}$, for any $x\in\ob(R)$. This defines an action of $G$ on $\MOD(R)$. It is easy to see that the map $f\mapsto {}^{g}f$ defines isomorphism of vector spaces ${}_{R}(M,N)\cong{}_{R}({}^{g}M,{}^{g}N)$.

Assume that $R$ is a locally bounded $K$-category and $X,Y\in\mod(R)$. In the Galois covering theory one frequently uses an observation that there are only finitely many elements $g\in G$ such that ${}_{R}(X,{}^{g}Y)\neq 0$ and ${}_{R}({}^{g}X,Y)\neq 0$ (this is a straightforward consequence of the assumption that $G$ acts freely on the objects of $R$). Applying this and the fact that $(F_{\lambda},F_{\bullet})$ is an adjoint pair one can show that the three bifunctors ${}_{A}(F_{\lambda}(-),F_{\lambda}(\cdot))$, $\bigoplus_{g\in G}{}_{R}(-,{}^{g}(\cdot))$ and $\bigoplus_{g\in G}{}_{R}({}^{g}(-),\cdot)$ are equivalent and this equivalence is induced by the push-down functor $F_{\lambda}:\mod(R)\ra\mod(A)$. In particular, $F_{\lambda}$ induces natural isomorphisms $$\nu_{X,Y}:\bigoplus_{g\in G}{}_{R}({}^{g}X,Y)\ra{}_{A}(F_{\lambda}(X),F_{\lambda}(Y))$$ of vector spaces, for any $X,Y\in\mod(R)$, given by $$\nu_{X,Y}((f_{g})_{g\in G})=\sum_{g\in G}F_{\lambda}(f_{g})$$ where $f_{g}:{}^{g}X\ra Y$, for any $g\in G$. In this description we identify $F_{\lambda}({}^{g}X)$ with $F_{\lambda}(X)$, for any $R$-module $X$ and $g\in G$. This isomorphism is used freely in the paper.

The \textit{support} $\supp (M)$ of a module $M\in\MOD(R)$ is the full subcategory of $R$ formed by all objects $x$ in $R$ such that $M(x)\neq 0$. The category $R$ is \textit{locally support-finite} \cite{DoSk} if and only if for any $x\in\ob(R)$ the union of the sets $\supp(M)$, where $M\in\ind(R)$ and $M(x)\neq 0$, is finite.

We say that the group $G$ is \textit{admissible} if and only if $G$ acts freely on the objects of $R$ and there are only finitely many $G$-orbits. In this case the orbit category $R\slash G$ is finite and we often treat it as an algebra. If $G$ is admissible, then we say that $G$ \textit{acts freely on $\ind(R)$} if and only if ${}^{g}M\cong M$ implies that $g=1$, for any $M\in\ind(R)$ and $g\in G$. 

The main properties of the push-down functor $F_{\lambda}:\mod(R)\ra\mod(A)$ are summarized in the following theorem, based on \cite{Ga,BoGa,MP,DoSk}.

\begin{thm}\label{0t1} Assume that $R$ is a locally bounded $K$-category, $G$ an admissible group of $K$-linear automorphisms of $R$ and $F:R\ra A$ the Galois covering. Then the functor $F_{\lambda}:\mod(R)\ra\mod(A)$ satisfies the following assertions.
\begin{enumerate}[\rm(1)]
 \item 	There are isomorphisms $F_{\lambda}({}^{g}M)\cong F_{\lambda}(M)$ and $F_{\lambda}({}^{g}f)\cong F_{\lambda}(f)$, for any $R$-module $M$, $R$-homomorphism $f$ and $g\in G$.
\item There is an isomorphism $F_{\bullet}(F_{\lambda}(M))\cong\bigoplus_{g\in G}{}^{g}M$, for any $R$-module $M$. Moreover, if $X,Y\in\ind(R)$, then $F_{\lambda}(X)\cong F_{\lambda}(Y)$ implies $Y\cong {}^{g}X$, for some $g\in G$.
    \item If the group $G$ is torsion-free, then $G$ acts freely on $\ind(R)$. If the latter condition holds, then $F_{\lambda}:\mod(R)\ra\mod(A)$ preserves indecomposability.
\item The functor $F_{\lambda}$ induces the following isomorphisms of vector spaces $$\bigoplus_{g\in G}{}_{R}({}^{g}X,Y)\cong{}_{A}(F_{\lambda}(X),F_{\lambda}(Y))\cong\bigoplus_{g\in G}{}_{R}(X,{}^{g}Y),$$ for any $X,Y\in\mod(R)$.
	\item Assume that the group $G$ is torsion-free and $R$ is locally support-finite. The the push-down functor $F_{\lambda}:\mod(R)\ra\mod(A)$ is dense. This means that for any $M\in\mod(A)$ there is $X\in\mod(R)$ such that $F_{\lambda}(X)\cong M$.
\end{enumerate}
\end{thm}

We recall that if the group $G$ is torsion-free, then the functor $F_{\lambda}:\mod(R)\ra\mod(A)$ preserves right and left minimal almost split homomorphisms, Auslander-Reiten sequences and induces an injection $\ind(R)\slash G\hookrightarrow\ind(A)$, see \cite{BoGa,DoSk}.

In case $G$ is torsion free, we say that $F_{\lambda}:\mod(R)\ra\mod(A)$ is a \emph{Galois $G$-precovering of module categories}. If additionally $F_{\lambda}$ is dense, for example when $R$ is locally support-finite, we say that $F_{\lambda}:\mod(R)\ra\mod(A)$ is a \emph{Galois $G$-covering of module categories}. Similar terminology is used for functors between any additive Krull-Schmidt $K$-categories if they satisfy analogous conditions. We refer to \cite[Definition 2.3]{P6'} for the precise formulation. 

The push-down functor $F_{\lambda}:\mod(R)\ra\mod(A)$ is not dense in general. This is the source of the following definition, originated in \cite{DoSk}. 

\begin{df}\label{d1&2} Assume that $F:R\ra A$ is a Galois $G$-covering and $G$ is torsion-free. An indecomposable module $M\in\mod(A)$ is of the \emph{first kind} if and only if $M$ lies in the image of the functor $F_\lambda:\mod(R)\ra\mod(A)$. Otherwise, $M$ is of the \emph{second kind}.
\end{df} We emphasize that \cite{DoSk} gives important characterizations of the modules of the first and the second kind for a wide class of Galois coverings.

A particularly important case of a Galois covering $F:R\ra A$ is the situation when $R$ is \emph{locally representation-finite}. This means that for any $x\in\ob(R)$ there are only finitely many indecomposable modules $M\in\mod(R)$ such that $M(x)\neq 0$. In this context, the following classical result is useful, see \cite{Ga,BoGa,BrGa,MP}. We apply it in Section 4.

We shall denote by $[X]$ the vertices of the Auslander-Reiten quiver $\Gamma_R$, that is, the isomorphism classes of indecomposable finite dimensional $R$-modules $X$. 

\begin{thm}\label{t10} Assume that $R$ is a locally bounded $K$-category, $G$ an admissible torsion-free group $K$-linear automorphisms of $R$ and $F:R\ra A$ the associated Galois covering. The following assertions hold.
\begin{enumerate}[\rm(1)]
 \item The category $R$ is locally representation-finite if and only if the category $\ind(R)$ is locally bounded. In this case, the category $A$ is representation-finite.
	\item Assume that $R$ is locally representation-finite. In this case, the push-down functor $F_{\lambda}:\ind(R)\ra\ind(A)$ is a Galois covering of locally bounded $K$-categories which induces a Galois covering $\Gamma_{R}\ra\Gamma_{A}$ of translation quivers such that $[X]\mapsto[F_{\lambda}(X)]$, for any $X\in\ind(R)$.
\end{enumerate}
\end{thm}

\begin{proof} Theorem follows mainly from \ref{0t1}. We only show the first part of $(1)$, because it is very instructive. First observe that $\ind(R)$ is locally bounded if and only if the hom-functors ${}_{R}(*,X)$ and ${}_{R}(X,*)$ have finite supports, for any $X\in\ind(R)$. 

Assume that $R$ is locally representation-finite and let $X\in\ind(R)$. We show that the functor ${}_{R}(*,X)$ has finite support. Let $R_a$ be the finite set of all indecomposable modules in $\mod(R)$ which are nonzero in $a\in\ob(R)$. Further, let $\CX$ be the union of all the sets $R_{a}$ such that $a\in\supp(X)$. If $Y\in\ind(R)$ and ${}_{R}(Y,X)\neq 0$, then $\supp(X)\cap\supp(Y)\neq\emptyset$ and hence $Y\in\CX$. Since $\CX$ is finite, we conclude that ${}_{R}(*,X)$ has finite support and thus $\ind(R)$ is locally bounded. Arguments for the functor ${}_{R}(X,*)$ are analogous.

Assume now that the hom-functors have finite supports. We prove that $R$ is locally representation-finite. Let $a\in\ob(R)$ and recall that $X(a)\cong{}_{R}(P_a,X)$, for any indecomposable module $X\in\mod(R)$, and hence $X(a)\neq 0$ if and only if ${}_{R}(P_a,X)\neq 0$. Since the support of the functor ${}_{R}(P_a,*)$ is finite, we conclude that there is a finite number of indecomposable $R$-modules $X\in\mod(R)$ with $X(a)\neq 0$. Thus $R$ is locally representation-finite.
\end{proof}

\section{Covering theory of functor categories}

In this section we outline the main results from \cite[Section 3]{P6'}. First we recall some basic facts about functor categories and general tensor products for categories of modules over small $K$-categories. Then we introduce \emph{Galois covering theory of functor categories}. We view this theory as the theory of the left and the right adjoint functors $$\Phi,\Theta:\MOD(\mod(R))\ra\MOD(\mod(A))$$ to the pull-up functor $$\Psi=(F_{\lambda})_{\bullet}:\MOD(\mod(A))\ra\MOD(\mod(R)),$$ along the push-down functor $F_{\lambda}\colon\mod(R)\ra\mod(A)$ where $(F_{\lambda})_{\bullet}=(-)\circ F_{\lambda}$. It turns out that when $F_{\lambda}$ is dense, $\Phi$ and $\Theta$ are natural generalizations of  $F_{\lambda}$ and $F_{\rho}$ to the level of functor categories. Generally, $\Phi$ and $\Theta$ restrict to categories $\CF(R),\CF(A)$ of finitely presented functors and the restricted functors $\Phi,\Theta:\CF(R)\ra\CF(A)$ coincide. This is important for applications to the theory of Krull-Gabriel dimension, see Section 5. Finally, we show that $\Phi:\CF(R)\ra\CF(A)$ is a Galois $G$-precovering of functor categories if the group $G$ is torsion-free.

\subsection{Functor categories and tensor products}

Assume that $R$ is a locally bounded $K$-category. Set $\CR=\mod(R)$ and denote by $\mod(\CR)$ the category of all contravariant $K$-linear functors $\mod(R)\ra\mod(K)$.

Let $M$ be an $R$-module. Then a \textit{contravariant hom-functor} represented by $M$ is the functor $H_{M}:\mod(R)\ra\MOD(K)$ such that $H_{M}(X)={}_{R}(X,M)$, for any $X\in\mod(R)$, and if $f\in{}_{R}(X,Y)$, then $H_{M}(f):{}_{R}(Y,M)\ra{}_{R}(X,M)$ where $H_{M}(f)(g)=gf$, for any $g\in{}_{R}(Y,M)$. The functor $H_{M}:\mod(R)\ra\MOD(K)$ is denoted by ${}_{R}(-,M)$, but $-$ may be replaced by $*,?$ etc. We agree that the domain of a hom-functor can be enlarged to $\Mod(R)$ or $\MOD(R)$.

Assume that $f\in{}_{R}(M,N)$ is an $R$-homomorphism. Then $f$ induces a homomorphism of hom-functors ${}_{R}(-,f):{}_{R}(-,M)\ra{}_{R}(-,N)$ such that ${}_{R}(X,f):{}_{R}(X,M)\ra{}_{R}(X,N)$ is defined by ${}_{R}(X,f)(g)=fg$, for any $g\in{}_{R}(X,M)$. The Yoneda lemma implies that the function $f\mapsto{}_{R}(-,f)$ defines an isomorphism $${}_{R}(M,N)\ra{}_{\CR}({}_{R}(-,M),{}_{R}(-,N))$$ of vector spaces. Moreover, this yields $M\cong N$ if and only if ${}_{R}(-,M)\cong{}_{R}(-,N)$. 

A functor $F\in\mod(\CR)$ is \textit{finitely generated} if and only if there exists an epimorphism of functors ${}_{R}(-,N)\ra F$, for some $N\in\mod(R)$. Furthermore, $F$ is \textit{finitely presented} if and only if there exists an exact sequence of functors $${}_{R}(-,M)\xrightarrow{{}_{R}(-,f)}{}_{R}(-,N)\ra F\ra 0,$$ for some $M,N\in\mod(R)$ and $R$-module homomorphism $f:M\ra N$. In this situation we have $F\cong\Coker{}_{R}(-,f)$ and thus $F(X)$ is isomorphic to the cokernel of the map ${}_{R}(X,f):{}_{R}(X,M)\ra{}_{R}(X,N)$.

The full subcategory of $\mod(\CR)$, formed by all finitely presented functors, is denoted as $\CF(R)$. Observe that ${}_{R}(-,M)\in\CF(R)$, for any $M\in\mod(R)$ and recall that hom-functors are projective objects of the category $\CF(R)$ \cite{Au0}.

The following result is well-known, see for example \cite[Proposition 3.4]{P6'}.

\begin{prop}\label{p3} Assume that $R$ is a locally bounded $K$-category. The category $\CF(R)$ is an abelian Krull-Schmidt hom-finite $K$-category. Moreover, a functor $T\in\CF(R)$ has local endomorphism $K$-algebra if and only if $T$ is indecomposable.
\end{prop}

Assume that $\CA,\CB$ are small $K$-categories, in particular $\CA,\CB$ may be abelian. We denote by $$\CA(-,-):\CA^{\op}\times\CA\ra\MOD(K)$$ the \emph{hom-bifunctor}. We denote by $\CA\lMOD$ the category of all \emph{left $\CA$-modules}, that is, the category consisting of covariant $K$-linear functors $\CA\ra\MOD(K)$ as objects and natural transformations as morphisms. The category of all \emph{right $\CA$-modules}, denoted by $\MOD(\CA)$, is the category consisting of covariant $K$-linear functors $\CA^{\op}\ra\MOD(K)$ (equivalently, contravariant functors $\CA\ra\MOD(K)$) as objects and natural transformations as morphisms. We often identify $\MOD(\CA^{\op})$ with $\CA\lMOD$, because a left $\CA$-module is a right $\CA^{\op}$-module. An \emph{$\CA$-$\CB$-bimodule} is a functor $\CB^{\op}\times\CA\ra\MOD(K)$, that is, a functor which is contravariant in the first variable and covariant in the second.

Objects of $\CA$ are usually denoted with lowercase letters and objects of $\MOD(\CA)$ or $\CA\lMOD$ with uppercase letters. Observe that $\CA(-,-)$ is an $\CA$-$\CA$-bimodule and if $a\in\CA$, then $\CA(a,-)\in\CA\lMOD$ and $\CA(-,a)\in\MOD(A)$. We also denote by $${}_{\CA}(-,-):\MOD(\CA)^{\op}\times\MOD(\CA)\ra\MOD(K)$$ the \emph{hom-bifunctor}. Clearly ${}_{\CA}(-,-)=\MOD(\CA)(-,-)$ but we do not use the latter notation. Observe that ${}_{\CA^{\op}}(-,-)$ is the hom-bifunctor defined for the category of left $\CA$-modules. The notation introduced above is consistent with that for modules over locally bounded $K$-categories.

We recall from \cite{FPN,Mi} that there exists a \emph{tensor product bifunctor} $$-\otimes_{\CA}-:\MOD(\CA)\times\CA\lMOD\ra\MOD(K)$$ such that, for any $\CA$-$\CB$-bimodule ${}_{\CA}M_{\CB}$ and $\CB$-$\CA$-bimodule ${}_{\CB}N_{\CA}$:
\begin{itemize}
  \item the functor $-\otimes{}_{\CA}M_{\CB}:\MOD(\CA)\ra\MOD(\CB)$ is the left adjoint to the functor ${}_{\CB}({}_{\CA}M_{\CB},-):\MOD(\CB)\ra\MOD(\CA)$,
  \item the functor ${}_{\CB}N_{\CA}\otimes_{\CA}-:\CA\lMOD\ra\CB\lMOD$ is the left adjoint to the functor ${}_{\CB^{\op}}({}_{\CB}N_{\CA},-):\CB\lMOD\ra\CA\lMOD$.
\end{itemize} It is known that $$M\otimes_{\CA}\CA(-,-)\cong M\textnormal{ and }\CA(-,-)\otimes_{\CA} N\cong N,$$ for any $M\in\MOD(\CA)$, $N\in\CA\lMOD$. In particular, we have natural isomorphisms $$M\otimes_{\CA}\CA(a,-)\cong M(a)\textnormal{ and }\CA(-,a)\otimes_{\CA} N\cong N(a),$$ for any $a\in\CA$. These isomorphisms are called the \emph{co-Yoneda isomorphisms}. As the left adjoints to the appropriate hom-functors, the tensor product functors $-\otimes{}_{\CA}M_{\CB}$ and ${}_{\CB}N_{\CA}\otimes_{\CA}-$ are unique up to natural equivalence. Hence we conclude that for locally bounded $K$-categories $\CA$ and $\CB$ these functors coincide with the usual tensor products for modules.

\subsection{Covering theory}

The \emph{pull-up functor} along a $K$-linear functor $G:\CA\ra\CB$ is the functor $$G_{\bullet}:\MOD(\CB)\ra\MOD(\CA)$$ defined as $(-)\circ G^{\op}$. Recall that this functor is often called the \emph{restriction functor}, but we do not use this terminology, see \cite[Remark 3.2]{P6'}. 

Let $F:R\ra A$ be a Galois $G$-covering and denote by $$\Psi:=(F_{\lambda})_{\bullet}:\MOD(\CA)\ra\MOD(\CR)$$ be the pull-up functor along $F_{\lambda}:\mod(R)\ra\mod(A)$. Applying the tensor product bifunctor we obtain the following result, describing the left adjoint and the right adjoint to $\Psi$, see \cite[Corollary 3.4, Theorem 3.1]{P6'}. We believe that this fact initiates the \emph{general Galois covering theory of functor categories}. In the sequel we set $\CR=\mod(R)$ and $\CA=\mod(A)$. 

\begin{thm}\label{c1} Assume that $F:R\ra A$ is a Galois $G$-covering. The functor $$\Phi:=?\otimes_{\CR}[{}_{A}(-,F_\lambda(*))]:\MOD(\CR)\ra\MOD(\CA)$$ is the left adjoint to $\Psi=(F_{\lambda})_{\bullet}$ and the functor $$\Theta:={}_{\CR}({}_{A}(F_{\lambda}(*),-),?):\MOD(\CR)\ra\MOD(\CA)$$ is the right adjoint to $\Psi$.
\end{thm}

\begin{proof} We recall from \cite[Theorem 3.1]{P6'}, see also \cite[Proposition 6.1]{Bu}, that for any $K$-linear functor $G:\CA\ra\CB$, the pull-up functor $G_{\bullet}:\MOD(\CB)\ra\MOD(\CA)$ has the left adjoint given by $$G_{L}=?\otimes_{\CA}\CB(-,G(*)):\MOD(\CA)\ra\MOD(\CB)$$ and the right adjoint given by $$G_{R}={}_{\CA}(\CB(G(*),-),?):\MOD(\CA)\ra\MOD(\CB).$$ Note that the bifunctor ${}_{A}(-,F_\lambda(*))$ is a $\CR$-$\CA$-bimodule, the bifunctor ${}_{A}(F_{\lambda}(*),-)$ is an $\CA$-$\CR$-bimodule and clearly $\Phi=(F_{\lambda})_L$, $\Theta=(F_{\lambda})_R$. This shows the thesis. 
\end{proof}

The following result, proved originally in \cite[Theorem 3.7]{P6'}, shows important properties of the functors $\Phi,\Theta:\MOD(\CR)\ra\MOD(\CA)$. Moreover, this is a crucial ingredient of the proof of the fact that Galois coverings do not increase Krull-Gabriel dimension \cite[Theorem 3.8]{P6'}. 

\begin{thm}\label{t2} Assume that $F:R\ra A$ is a Galois $G$-covering. Then $\Phi(\CF(R))\subseteq\CF(A)$, $\Theta(\CF(R))\subseteq\CF(A)$ and $\Phi|_{\CF(R)}\cong\Theta|_{\CF(R)}$. In particular, the functors $\Phi,\Theta:\CF(R)\ra\CF(A)$ are exact. Moreover, we have $$\Phi(\Coker_{R}(*,f))\cong\Theta(\Coker_{R}(*,f))\cong\Coker_{A}(-,F_{\lambda}(f)),$$ for any $R$-module homomorphism $f$. 
\end{thm}

\begin{proof} The co-Yoneda isomorphisms yields $$\Phi({}_{R}(*,?))={}_{R}(*,?)\otimes_{\CR}[{}_{A}(-,F_\lambda(*))]\cong{}_{A}(-,F_\lambda(?)).$$ Moreover, $\Phi:\MOD(\CR)\ra\MOD(\CA)$ is right exact, because it is the left adjoint. Hence we obtain $\Phi(\Coker_{R}(*,f))\cong\Coker_{A}(-,F_{\lambda}(f))$ which shows that $\Phi(\CF(R))\subseteq\CF(A)$. We show that there is a natural equivalence of functors $$\varphi=(\varphi^{T})_{T\in\CF(R)}:\Phi|_{\CF(R)}\ra\Theta|_{\CF(R)}$$ where $\varphi^{T}:\Phi(T)\ra\Theta(T)$, for any $T\in\CF(R)$. Our aim is to sketch the definition of $\varphi^{T}:\Phi(T)\ra\Theta(T)$. Assume that $f:M\ra N\in\mod(R)$ and set $T=\Coker_{R}(*,f)$, that is, we have an exact sequence of functors of the form $${}_{R}(*,M)\xrightarrow{{}_{R}(*,f)}{}_{R}(*,N)\xrightarrow{\pi} T\ra 0.$$ Then we obtain $$\Phi(T)\cong\Coker_{A}(-,F_{\lambda}(f))\cong \Coker_{A}(-,F_{\rho}(f))\cong\Coker_{R}(F_{\bullet}(-),f)$$ since $F_{\lambda}=F_{\rho}$ on $\mod(R)$, and moreover $$\Theta(T)={}_{\CR}({}_{A}(F_{\lambda}(*),-),T)\cong {}_{\CR}({}_{R}(*,F_{\bullet}(-)),T).$$ Observe that $F_{\bullet}(\mod(A))\subseteq\Mod(R)$ so we cannot conclude that $\Coker_{R}(F_{\bullet}(-),f)\cong T(F_{\bullet}(-))$ nor can we apply the Yoneda lemma to state that ${}_{\CR}({}_{R}(*,F_{\bullet}(-)),T)\cong T(F_{\bullet}(-))$. Then we define $$\varphi^{T}=(\varphi^{T}_{X})_{X\in\mod(A)}:\Coker_{R}(F_{\bullet}(-),f)\ra{}_{\CR}({}_{R}(*,F_{\bullet}(-)),T)$$ in the following way. Assume that $X\in\mod(A)$ and $\alpha:F_{\bullet}(X)\ra N$ is an $R$-homomorphism. The $K$-linear map $$\varphi^{T}_{X}:\Coker_{R}(F_{\bullet}(X),f)\cong\frac{{}_{R}(F_{\bullet}(X),N)}{\Im{}_{R}(F_{\bullet}(X),f)}\ra{}_{\CR}({}_{R}(*,F_{\bullet}(X)),T)$$ is defined by the formula: $$\varphi^{T}_{X}(\alpha+\Im{}_{R}(F_{\bullet}(X),f))=\pi\circ{}_{R}(*,\alpha).$$ 

It is shown in \cite[Theorem 3.7]{P6'} that $\varphi^{T}_{X}$ is a well-defined $K$-linear isomorphism, for any $X\in\mod(A)$, $\varphi^{T}$ is a natural equivalence, for any $T\in\CF(R)$, and finally that $\varphi:\Phi|_{\CF(R)}\ra\Theta|_{\CF(R)}$ is a natural equivalence. This yields that $\Theta(\CF(R))\subseteq\CF(A)$ since $\Phi(\CF(R))\subseteq\CF(A)$, as shown above. In consequence, the functors $\Phi,\Theta:\CF(R)\ra\CF(A)$ are exact as they are both left and right exact.
\end{proof}


We describe the action of $G$ on $\MOD(\CR)$. Given a functor $T\in\MOD(\CR)$ and $g\in G$ we define $gT\in\MOD(\CR)$ to be as follows: $(gT)(X)=T({}^{g^{-1}}X)$ and $(gT)(f)=T({}^{g^{-1}}f)$, for any module $X\in\mod(R)$ and homomorphism $f\in\mod(R)$. Given a morphism of functors $\iota:T_{1}\ra T_{2}$ and $g\in G$ we define $g\iota:gT_{1}\ra gT_{2}$ as $(g\iota)_{X}=\iota_{X^{-g}}$, for any $X\in\mod(R)$. It is easy to see that this is an action of $G$. Moreover, the action restricts to $\CF(R)$, because $$g(\Coker{}_{R}(*,f))\cong\Coker{}_{R}(*,{}^{g}f)\in\CF(R).$$ If $G$ is torsion-free, then $G$ acts freely on $\CF(R)$ \cite[Proposition 4.1]{P6'}.


The following theorem shows fundamental properties of the functor $\Phi:\CF(R)\ra\CF(A)$ \cite[Theorem 4.3]{P6'}. This properties particularly imply that $\Phi$ is a Galois $G$-precovering of functor categories, provided that $G$ is torsion-free. The torsion-freeness of $G$ is necessary to show that $\Phi$ preserves indecomposable functors, but the remaining properties are independent of this assumption. 

\begin{thm}\label{t4} Assume that $R$ is a locally bounded $K$-category, $G$ an admissible group of $K$-linear automorphisms of $R$ and $F:R\ra A\cong R\slash G$ the Galois $G$-covering. The functor $\Phi:\CF(R)\ra\CF(A)$ has the following properties.

\begin{enumerate}[\rm(1)]
	\item There are isomorphisms $\Phi(T)\cong\Phi(gT)$ and $\Psi(\Phi(T))\cong\bigoplus_{g\in G}gT$, for any functor $T\in\CF(R)$ and $g\in G$. Moreover, if functors $T_{1},T_{2}\in\CF(R)$ are indecomposable, then $\Phi(T_{1})\cong\Phi(T_{2})$ implies $T_{1}\cong gT_{2}$, for some $g\in G$.
\item If $G$ is torsion-free, then $\Phi$ preserves indecomposability.
	\item Assume $T,T'\in\CF(R)$ and $U\in\CF(A)$. There are natural isomorphisms of vector spaces $${}_{\CA}(\Phi(T),U)\cong {}_{\CR}(T,\Psi(U))\textnormal{ and }{}_{\CR}(\Psi(U),T)\cong {}_{\CA}(U,\Phi(T))$$ which induce natural isomorphisms $$\bigoplus_{g\in G}{}_{\CR}(gT,T')\cong{}_{\CA}(\Phi(T),\Phi(T'))\cong\bigoplus_{g\in G}{}_{\CR}(T,gT').$$
\end{enumerate} Consequently, if the group $G$ is torsion-free, then the functor $\Phi:\CF(R)\ra\CF(A)$ is a Galois $G$-precovering of functor categories.
\end{thm}

\begin{proof}
(1) Assume $g\in G$ and $T=\Coker{}_{R}(*,f)$, for some $R$-homomorphism $f:M\ra N$. Then $gT\cong\Coker{}_{R}(*,{}^{g}f)$  and since $F_{\lambda}(f)\cong F_{\lambda}({}^{g}f)$, we obtain $$\Phi(gT)\cong\Coker{}_{A}(-,F_{\lambda}({}^{g}f))\cong\Coker{}_{A}(-,F_{\lambda}(f))\cong\Phi(T).$$ Since $F_{\bullet}(F_{\lambda}(f))\cong\bigoplus_{g\in G}{}^{g}f$, we obtain the following natural isomorphisms: $$\Psi(\Phi(T))=\Phi(T)\circ F_{\lambda}\cong\Coker_{R}(F_{\lambda}(*),F_{\lambda}(f))\cong\Coker_{R}(*,F_{\bullet}(F_{\lambda}(f)))\cong$$$$\cong
\Coker_{R}(*,\bigoplus_{g\in G}{}^{g}f)\cong\Coker_{R}(\bigoplus_{g\in G}(*,{}^{g}f))\cong\bigoplus_{g\in G}\Coker_{R}(*,{}^{g}f)\cong\bigoplus_{g\in G}gT.$$ To show the second part, assume that $T_{1},T_{2}\in\CF(R)$ are indecomposable functors and $\Phi(T_{1})\cong\Phi(T_{2})$. Then $\Psi(\Phi(T_{1}))\cong\Psi(\Phi(T_{2}))$ and hence $\bigoplus_{g\in G}gT_{1}\cong\bigoplus_{g\in G}gT_{2}$. Moreover, $T_{1}$ and $T_{2}$ have local endomorphism algebras and thus $gT_{1}$ and $gT_{2}$ as well, for any $g\in G$ (observe that $\End_{\CR}(T)\cong\End_{\CR}(gT)$, for any $T\in\CF(R)$). Hence the Krull-Schmidt theorem yields $T_{1}\cong gT_{2}$, for some $g\in G$.

(3) Assume that $G$ is torsion-free and $T\in\CF(R)$ is indecomposable. Assume that $\Phi(T)\cong U\oplus V$, for some functors $U,V\in\CF(A)$ such that $U$ is indecomposable. Then $\bigoplus_{g\in G}gT\cong\Psi(U)\oplus\Psi(V)$ and so the Krull-Schmidt theorem yields $\Psi(U)\cong\bigoplus_{h\in H}hT$, for some nonempty set $H\subseteq G$. Since $F_{\lambda}$ is $G$-invariant, we get $g\Psi(U)=g(U\circ F_{\lambda})\cong U\circ F_{\lambda}$ and thus $$g\Psi(U)\cong\bigoplus_{h\in H}(gh)T\cong\bigoplus_{h\in H}hT,$$ for any $g\in G$. Since $G$ acts freely on $\CF(R)$, we conclude that $gH\subseteq H$, for any $g\in G$ which implies that $H=G$. Thus $V=0$ and so $\Phi(T)\cong U$ is indecomposable.

(4) The first two natural isomorphisms follow from the fact that $(\Phi,\Psi)$ and $(\Psi,\Theta)$ are adjoint pairs and $\Phi|_{\CF(R)}\cong\Theta|_{\CF(R)}$. Then we obtain $${}_{\CA}(\Phi(T),\Phi(T'))\cong {}_{\CR}(T,\Psi(\Phi(T')))\cong{}_{\CR}(T,\bigoplus_{g\in G}{}gT').$$ Observe that there are only finitely many $g\in G$ such that ${}_{\CR}(T,gT')\neq 0.$ Indeed, there are epimorphisms ${}_{R}(*,N)\ra T$ and ${}_{R}(*,N)\ra T'$, because $T,T'$ are finitely generated. Hence ${}_{\CR}(T,gT')\neq 0$ yields $${}_{\CR}({}_{R}(*,N),{}_{R}(*,{}^{g}N'))\cong{}_{R}(N,{}^{g}N')\neq 0$$ and ${}_{R}(N,{}^{g}N')\neq 0$ holds only for finite number of $g\in G$. It is easy to see that ${}_{\CR}(T,gT')\cong{}_{\CR}(g^{-1}T,T')$, for any $g\in G$ and thus we obtain $${}_{\CR}(T,\bigoplus_{g\in G}{}gT')\cong\bigoplus_{g\in G}{}_{\CR}(T,gT')\cong\bigoplus_{g\in G}{}_{\CR}(gT,T').$$ \end{proof} 

We show in the proof of \cite[Theorem 4.3]{P6'} that the isomorphisms $$\bigoplus_{g\in G}{}_{\CR}(gT,T')\cong{}_{\CA}(\Phi(T),\Phi(T'))\cong\bigoplus_{g\in G}{}_{\CR}(T,gT')$$ are given by $$(\iota_{g})_{g\in G}\mapsto\sum_{g\in G}\Phi(\iota_{g})\mapsfrom(g^{-1}\iota_{g})_{g\in G}$$ where $\iota_{g}:gT_{1}\ra T_{2}$ is a homomorphism of functors, for any $g\in G$. Hence these isomorphisms are induced by $\Phi$. Recall that this requirement is necessary for $\Phi$ to be a Galois $G$-precovering. We omit a technical proof this fact.

The above theorem is a generalization of \cite[Theorem 5.5]{P4} where it is shown that the functor $\Phi:\CF(R)\ra\CF(A)$ is a Galois $G$-precovering under the assumption that $R$ is locally support-finite. In this case the push-down functor $F_{\lambda}:\mod(R)\ra\mod(A)$ is dense which was significant for the proof. 

The dense case is important, because then the functor $F_{\lambda}:\mod(R)\ra\mod(A)$ itself is a Galois covering of module categories, see Theorem \ref{0t1}. In this special situation, description of the functors $\Phi,\Theta:\MOD(\CR)\ra\MOD(\CA)$ is analogous to the classical push-down functors $F_{\lambda},F_{\rho}:\MOD(R)\ra\MOD(A)$. In particular, we have $$\Phi(T)(F_{\lambda}(M))\cong\bigoplus_{g\in G}T({}^{g}M)\textnormal{ and }\Theta(T)(F_{\lambda}(M))\cong\prod_{g\in G}T({}^{g}M),$$ for any $T\in\MOD(\CR)$ and $M\in\mod(R)$. We refer to \cite[Corollary 4.6]{P6'} for details of this description. Nevertheless, in order to give a flavor of these results, we recall below an exact description of $\Phi:\MOD(\CR)\ra\MOD(\CA)$:

\begin{itemize}
  \item Assume that $\alpha:F_{\lambda}(M)\ra F_{\lambda}(N)$ is an $A$-homomorphism such that $\alpha=\sum_{g\in G}F_{\lambda}(f_{g})$ where $f_{g}:{}^{g}M\ra N$, for any $g\in G$. Then $$\Phi(T)(\alpha):\bigoplus_{g\in G}T({}^{g}N)\ra\bigoplus_{g\in G}T({}^{g}M)$$ is defined by $T({}^{h}f_{h^{-1}g}):T({}^{h}N)\ra T({}^{g}M)$, for any $g,h\in G$.
  \item Assume that $T_{1},T_{2}\in\MOD(\CR)$, $\iota:T_{1}\ra T_{2}$ is a morphism of functors and let $M\in\mod(R)$. Then the homomorphism $$\Phi(\iota)_{F_{\lambda}(M)}:\bigoplus_{g\in G}T_{1}({}^{g}M)\ra\bigoplus_{g\in G}T_{2}({}^{g}M)$$ is defined by $\iota_{{}^{g}M}:T_{1}({}^{g}M)\ra T_{2}({}^{g}M)$, for any $g\in G$.
\end{itemize} These properties essentially follow from the fact that there is a natural equivalence of functors $\Psi(\Phi(\cdot))\cong\bigoplus_{g\in G}g(\cdot)$, see \cite[Theorem 4.5]{P6'} (we have also $\Psi(\Theta(\cdot))\cong\prod_{g\in G}g(\cdot)$).

The above description suggests an equivalent definition of $\Phi:\MOD(\CR)\ra\MOD(\CA)$ when $F_{\lambda}:\mod(R)\ra\mod(A)$ is dense. Namely, denote by $\Add(\mod(R))$ the full subcategory of $\MOD(R)$ whose objects are arbitrary direct sums of finite dimensional $R$-modules. Assume that $T\in\MOD(\CR)$ and let $\wh{T}:\Add(\mod(R))\ra\MOD(K)$ be the additive closure of $T$. The density of $F_{\lambda}$ ensures that $F_{\bullet}(\mod(A))\subseteq\Add(\mod(R))$, because $F_{\bullet}(F_{\lambda}(N))\cong\bigoplus_{g\in G}{}^{g}N$. Consequently, the functor $\wh{(\cdot)}\circ F_{\bullet}:\MOD(\CR)\ra\MOD(\CA)$ is well-defined and it is easy to see that $\Phi\cong\wh{(\cdot)}\circ F_{\bullet}$. Indeed, this follows from $F_{\bullet}F_{\lambda}\cong(\cdot)^{g}$ and the fact that for any $\alpha:F_{\lambda}(M)\ra F_{\lambda}(N)$ as above $F_{\bullet}(\alpha):\bigoplus_{g\in G}{}^{g}M\ra\bigoplus_{g\in G}{}^{g}N$ is defined by homomorphisms ${}^{h}f_{h^{-1}g}:{}^{g}M\ra {}^{h}N$, for any $g,h\in G$.

As a consequence of above considerations, we obtain the following fact.

\begin{thm}\label{t6} Assume that the push-down functor $F_{\lambda}:\mod(R)\ra\mod(A)$ is dense. Then $\Phi:\MOD(\CR)\ra\MOD(\CA)$ is a subfunctor of $\Theta:\MOD(\CR)\ra\MOD(\CA)$. Moreover, these functors restrict to categories of finitely presented functors and they coincide on $\CF(R)$. Consequently, both functors $\Phi,\Theta:\CF(R)\ra\CF(A)$ are exact.
\end{thm}

\begin{proof} Observe that we have $$\Phi(F_{\lambda}(-))=\Phi(\cdot)\circ F_{\lambda}=\Psi(\Phi(\cdot))\cong\bigoplus_{g\in G}g(\cdot)$$ and moreover $$\Theta(F_{\lambda}(-))=\Theta(\cdot)\circ F_{\lambda}=\Psi(\Theta(\cdot))\cong\prod_{g\in G}g(\cdot),$$ hence $\Phi$ is a subfunctor of $\Theta$, because $\bigoplus_{g\in G}g(\cdot)$ is a subfunctor of $\prod_{g\in G}g(\cdot)$. If $T\in\CF(R)$ is finitely presented, then for any $M\in\mod(R)$ we have $T({}^{g}M)\neq 0$ only for finite number of $g\in G$. This yields $$\Phi(T)(F_{\lambda}(M))\cong\bigoplus_{g\in G}T({}^{g}M)\cong\prod_{g\in G}T({}^{g}M)\cong\Theta(T)(F_{\lambda}(M)),$$ so $\Phi(T)\cong\Theta(T)$ and thus $\Phi|_{\CF(R)}\cong\Theta|_{\CF(R)}$. Hence $\Phi$ and $\Theta$ are exact on $\CF(R)$.
\end{proof}

Summing up, we get an exact functor $\Phi:\CF(R)\ra\CF(A)$ such that $$\Phi(T)=\wh{T}\circ F_{\bullet}\cong\Coker{}_{A}(-,F_{\lambda}(f)),$$ for any $T=\Coker{}_{R}(-,f)\in\CF(R)$. We refer the reader to \cite[Remark 4.8]{P6'} for more detailed discussion of these properties.

\section{The functors of the first and the second kind} 

This section, based on \cite[Section 5]{P6'}, is devoted to motivating the introduction of the following terminology, originated in \cite{DoSk} for the case of modules, see Definition \ref{d1&2}. In the sequel the functor $\Phi:\CF(R)\ra\CF(A)$ is denoted by $\Phi_{F}$.

\begin{df}\label{d1} Assume that $F:R\ra A$ is a Galois $G$-covering with a torsion-free group $G$ and let $\Phi_{F}:\CF(R)\ra\CF(A)$ be the associated Galois $G$-precovering of functor categories. We call an indecomposable functor $U\in\CF(A)$ a \emph{functor of the first kind} if and only if $U$ lies in the image of $\Phi_{F}$. Otherwise, we call $U$ a \emph{functor of the second kind}.
\end{df}

Recall that categories of finitely presented functors are abelian Krull-Schmidt categories. This implies that the density of $\Phi_{F}:\CF(R)\ra\CF(A)$ is equivalent with the condition that any functor $U\in\CF(A)$ is of the first kind.

In order to motivate Definition \ref{d1} we show examples of Galois coverings $F:R\ra A$ for which the functor $\Phi:\CF(R)\ra\CF(A)$ is either dense or not. In these examples $R$ is a simply connected locally representation-finite locally bounded $K$-category. As we shall see, this situation simplifies considerations, because then the functor $\Phi:\CF(R)\ra\CF(A)$ becomes a Galois covering on the module level.

\subsection{The locally representation-finite case}

A locally bounded $K$-category $R$ is \emph{triangular} if and only if there is a presentation $R\cong\underline{KQ}\slash I$ with $Q$ triangular, i.e. $Q$ has no oriented cycles. Furthermore, $R$ is \textit{simply connected} \cite{AsSk3} provided that, for any presentation $R\cong\underline{KQ}\slash I$ as a bound quiver $K$-category: 
\begin{itemize}
  \item the quiver $Q$ is \textit{triangular},
  \item the fundamental group $\Pi_{1}(Q,I)$ is trivial, i.e. $\Pi_{1}(Q,I)=\{1\}$.
\end{itemize} The reader is referred to \cite{Ga} for the appropriate definitions. It follows from \cite{SkBC} that $R$ is simply connected if and only if $R$ is triangular and has no proper Galois coverings. If $R$ is locally representation finite, then $R$ is simply connected if and only if the fundamental group $\Pi_{1}(\Gamma_{R})$ of $\Gamma_{R}$ is trivial.

Assume that $R$ is a locally bounded $K$-category. The \textit{mesh-category} \cite{Rie} of the Auslander-Reiten quiver $\Gamma_{R}$ of $R$ is defined as $\underline{K(\Gamma_{R}^{\op})}\slash I$ where $I$ is the admissible ideal in the path category $\underline{K(\Gamma_{R}^{\op})}$ generated by all the mesh-relations appearing in $\Gamma_{R}^{\op}$. We say that $R$ is \textit{standard} if and only if there exists a $K$-linear equivalence of categories $\phi_{R}:\ind(R)\ra K(\Gamma_{R})$ such that $\phi_{R}(X)=[X]$, for any $X\in\ind(R)$. It follows from \cite{BrGa,BoGa} that $R$ is standard if and only if $R$ admits a simply connected Galois covering. In particular, if $R$ is simply connected, then $R$ is standard.

It is known that the Auslander-Reiten quiver $\Gamma_{R}$ of $R$ describes the factor category $\ind(R)\slash\rad^{\infty}_{R}$ where $\rad^{\infty}_{R}$ is the \emph{infinite radical} of the category $R$. This means that in general we only have a full and dense functor $\ind(R)\ra K(\Gamma_{R})$ such that $X\mapsto [X]$, for any $X\in\ind(R)$. If $R$ is locally representation-finite, then $\ind(R)$ is locally bounded and hence $\ind(R)$ has a presentation $\underline{KQ}\slash I$ as a bound quiver $K$-category. Then $Q=\Gamma_{R}^{\op}$ and $I$ contains all the mesh-relations. However $I$ may contain more relations, so in general $\ind(R)$ is not equivalent with the mesh category $K(\Gamma_{R})$. This holds if $R$ is standard and in the particular case when $R$ is simply connected.


Assume that $R$ is a locally representation-finite locally bounded $K$-category, $G$ an admissible torsion-free group of $K$-linear automorphisms of $R$ and $F:R\ra A$ is the Galois covering. It follows from Theorem \ref{t10} that the push-down functor $F_{\lambda}:\ind(R)\ra\ind(A)$ is a Galois covering of locally bounded $K$-categories. Therefore $F_\lambda$ induces the push-down functor $$(F_{\lambda})_{\lambda}:\mod(\ind(R))\ra\mod(\ind(A))$$ on the level of module categories. We briefly sketch arguments showing that $(F_{\lambda})_{\lambda}\cong\Phi_{F}$ in this case, see \cite[Proposition 5.4]{P6'} for details.

First observe that the functor $\eta_{R}:\mod(\ind(R))\ra\CF(R)$ of \emph{additive closure} is an exact equivalence of categories. Clearly, if $T\in\mod(\ind(R))$, then $\eta_{R}(T):\mod(R)\ra\mod(K)$ is defined as the unique contravariant $K$-linear functor such that $\eta_{R}(T)(X)=T(X)$, for any $X\in\ind(R)$. Observe that $\eta_{R}(\ind(R)(-,M))={}_{R}(-,M)$, for any $M\in\mod(R)$, and hence $\eta_{R}$ preserves projective objects. It is easy to see that $\eta_{R}$ is exact, so it preserves projective presentations and thus $\Im(\eta_{R})\subseteq\CF(R)$. The quasi-inverse of this functor sends a functor $T\in\CF(R)$ to the restriction $T|_{\ind(R)}$.

Assume that $T\in\mod(\ind(R))$. Then the minimal projective presentation $p$ of $T$ in $\mod(\ind(R))$ has the form $$p:\quad\ind(R)(-,M)\ra\ind(R)(-,N)\ra T\ra 0,$$ for some $M,N\in\mod(R)$. Since $(F_{\lambda})_{\lambda}$ is exact and preserves projectivity, we get that $$(F_{\lambda})_{\lambda}(p):\quad\ind(A)(-,F_{\lambda}(M))\ra\ind(A)(-,F_{\lambda}(N))\ra (F_{\lambda})_{\lambda}(T)\ra 0$$ is a projective presentation of the $\ind(A)$-module $(F_{\lambda})_{\lambda}(T)$ in $\mod(\ind(A))$. Applying the functors $\eta_{R}$ and $\eta_{A}$ to $p$ and $(F_{\lambda})_{\lambda}(p)$, respectively, we obtain the following two projective presentations: $$\eta_{R}(p):\quad{}_{R}(-,M)\ra{}_{R}(-,N)\ra\eta_{R}(T)\ra 0,$$ $$\eta_{A}((F_{\lambda})_{\lambda}(p)):\quad{}_{A}(-,F_{\lambda}(M))\ra{}_{A}(-,F_{\lambda}(N))\ra\eta_{A}((F_{\lambda})_{\lambda}(T))\ra 0$$ which shows that $\Phi_{F}(\eta_{R}(T))=\eta_{A}((F_{\lambda})_{\lambda}(T))$. Summing up, we obtain the following commutative diagram $$\xymatrix{\mod(\ind(R))\ar[rr]^{(F_{\lambda})_{\lambda}}\ar[d]^{\eta_{R}}&&{}\mod(\ind(A))\ar[d]^{\eta_{A}}\\\CF(R)\ar[rr]^{\Phi_{F}}&& \CF(A)}$$ and conclude that $(F_{\lambda})_{\lambda}\cong\Phi_{F}$. In case $R$ is additionally simply connected, our description of $\Phi_{F}$ gets even more handy. Indeed, when $R$ is simply connected, both $R$ and $A$ are standard and there are $K$-linear equivalences $\phi_{R}:\ind(R)\ra K(\Gamma_{R})$ and $\phi_{A}:\ind(R)\ra K(\Gamma_{A})$, sending indecomposable modules to the associated vertices in the Auslander-Reiten quivers. The push-down functor $F_{\lambda}:\ind(R)\ra\ind(A)$ induces a Galois covering $\Gamma_{R}\ra\Gamma_{A}$ of translation quivers such that $[X]\mapsto[F_{\lambda}(X)]$, for any $X\in\ind(R)$ \ref{t10} and hence a Galois covering $\Gamma_{R}^{\op}\ra\Gamma_{A}^{\op}$ of the opposite quivers. In turn, this induces a Galois covering $K(\Gamma_{R})\ra K(\Gamma_{A})$ of the associated mesh-categories which we denote $F_{\lambda}^{\Gamma}:K(\Gamma_{R})\ra K(\Gamma_{A})$. Since $F_{\lambda}^{\Gamma}([X])=[F_{\lambda}(X)]$, we get the commutative diagram of the form: $$\xymatrix{\ind(R)\ar[rr]^{F_{\lambda}}\ar[d]^{\phi_{R}}&& \ind(A)\ar[d]^{\phi_{A}}\\K(\Gamma_{R})\ar[rr]^{F_{\lambda}^{\Gamma}}&&{}K(\Gamma_{A}.)}$$ 

We summarize the above observations in the following result. We refer to \cite[Theorem 5.6]{P6'} for the details.

\begin{thm}\label{t11} Assume that $R$ is a locally representation-finite simply connected locally bounded $K$-category, $G$ an admissible torsion-free group of $K$-linear automorphisms of $R$ and $F:R\ra A$ is the Galois covering. Then the following diagram $$\xymatrix{\mod(K(\Gamma_{R}))\ar[rr]^{(F_{\lambda}^{\Gamma})_{\lambda}}\ar[d]^{\widetilde{\phi}_{R}}&&\mod(K(\Gamma_{A}))\ar[d]^{\widetilde{\phi}_{A}}\\\mod(\ind(R))\ar[rr]^{(F_{\lambda})_\lambda}\ar[d]^{\eta_R}&&\mod(\ind(A))\ar[d]^{\eta_{A}}\\\CF(R)\ar[rr]^{\Phi_{F}}&& \CF(A),}$$ where $\widetilde{\phi}_{R}=(-)\circ\phi_{R}$ and $\widetilde{\phi}_{A}=(-)\circ\phi_{A}$, is a commutative diagram whose columns are equivalences.
\end{thm}

\begin{proof} Clearly $\widetilde{\phi}_{R}=(-)\circ\phi_{R}$ and $\widetilde{\phi}_{A}=(-)\circ\phi_{A}$ are $K$-linear equivalences on the module level. Moreover, any equivalence of Galois coverings induces an equivalence of the associated push-down functors, and hence the following diagram $$\xymatrix{\mod(K(\Gamma_{R}))\ar[rr]^{(F_{\lambda}^{\Gamma})_{\lambda}}\ar[d]^{\widetilde{\phi}_{R}}&&\mod(K(\Gamma_{A}))\ar[d]^{\widetilde{\phi}_{A}}\\\mod(\ind(R))\ar[rr]^{(F_{\lambda})_{\lambda}}&&\mod(\ind(A))}$$ is commutative.
\end{proof} 

\subsection{Examples}

Theorem \ref{t11} can be applied to justify Definition \ref{d1}. Indeed, we conclude that the functor $\Phi_{F}$ is dense if and only if the functor $(F_{\lambda}^{\Gamma})_{\lambda}$ is dense. The latter can be verified by classical techniques of covering theory. We use this idea in Examples \ref{e1} and \ref{e2} (see Examples 5.7 and 5.8 in \cite{P6'}, respectively). The first one gives a Galois covering $F:R\ra A$ for which $\Phi_{F}$ is not dense. In the second one it is the opposite. We are grateful to Stanisław Kasjan for drawing our attention to the first example.

\begin{exa}\label{e1} Assume that $R=\underline{KQ_{R}}\slash I_{R}$ where $Q_{R}$ is the following quiver: $$
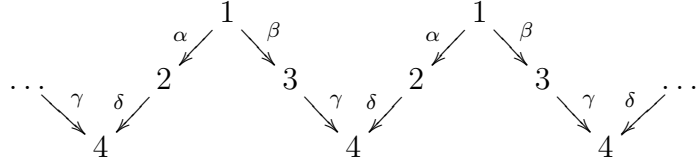
\xymatrix@C=1em@R=1em{&&&1\ar[dr]^{\beta}\ar[dl]_{\alpha}&&&&1\ar[dr]^{\beta}\ar[dl]_{\alpha}\\\dots\ar[dr]^{\gamma}&&2\ar[dl]_{\delta}&&3\ar[dr]^{\gamma}&&2\ar[dl]_{\delta}&&3\ar[dr]^{\gamma}&&\dots\ar[dl]_{\delta}\\&4&&&&4&&&&4}$$ and $I_{R}$ is generated by all zero relations of the form $\alpha\delta=0$ and $\beta\gamma=0$. The group $G=\mathbb{Z}$ acts on $R$ by the horizontal translation. Hence $R\slash G$ is the bound-quiver $K$-algebra $A=KQ_{A}\slash I_{A}$ where $Q_{A}$ is the quiver: $$
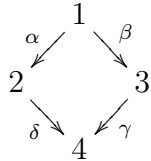
\xymatrix@C=1em@R=1em{&1\ar[dl]_{\alpha}\ar[dr]^{\beta}\\2\ar[dr]_{\delta}&&3\ar[dl]^{\gamma}\\&4}$$ and $I_{A}=\langle\alpha\delta,\beta\gamma\rangle$. We obtain a Galois covering functor $F:R\ra A$ with the covering group $G=\mathbb{Z}$. Observe that $Q_R$ is a quiver of type $\mathbb{A}^{\infty}_{\infty}$, so
$R$ is locally representation-finite simply connected. Hence $A$ is standard, but not simply connected since $\Pi_{1}(Q,I)=\ZZ$. 

Below we present the Auslander-Reiten quivers $\Gamma_{R},\Gamma_{A}$, one next to the other. Dotted lines on the quivers represent the existence of the Auslander-Reiten sequences. We see that $\Gamma_{R}$ contains a convex line $L$ whose vertices are marked as $\bullet$ and arrows are labelled by the letters $a,b,c,d$. We denote by $D$ the convex subquiver of $\Gamma_{A}$ of the type $\widetilde{\mathbb{A}}_{3}$ whose vertices and arrows are similarly marked on $\Gamma_{A}$.

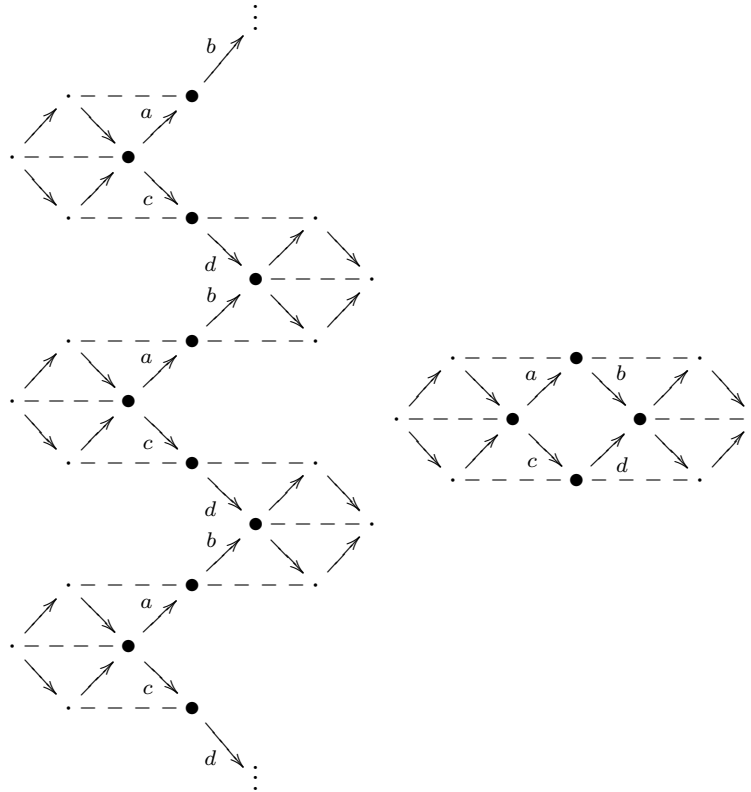
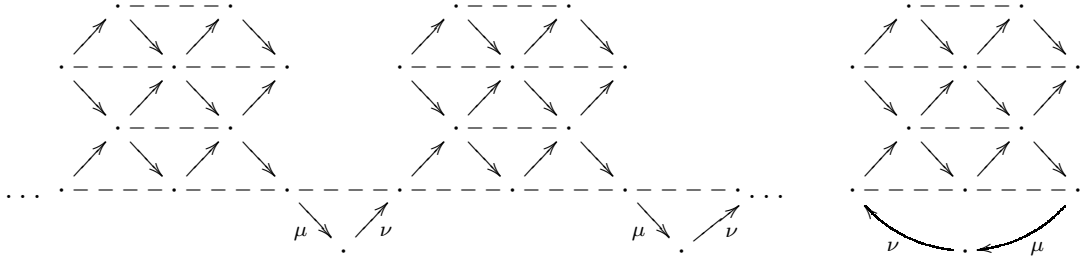
$$\xymatrix@C=1em@R=1em{&&&&\vdots\\&\cdot\ar[dr]\ar@{--}[rr]&&\bullet\ar[ur]^{b}\\\cdot\ar[ur]\ar[dr]\ar@{--}[rr]&&\bullet\ar[dr]_{c}\ar[ur]^{a}\\&\cdot\ar[ur]\ar@{--}[rr]&&\bullet\ar[dr]_{d}\ar@{--}[rr]&&\cdot\ar[dr]\\&&&&\bullet\ar[ur]\ar[dr]\ar@{--}[rr]&&\cdot\\&\cdot\ar[dr]\ar@{--}[rr]&&\bullet\ar[ur]^{b}\ar@{--}[rr]&&\cdot\ar[ur]\\\cdot\ar[ur]\ar[dr]\ar@{--}[rr]&&\bullet\ar[dr]_{c}\ar[ur]^{a}\\&\cdot\ar[ur]\ar@{--}[rr]&&\bullet\ar[dr]_{d}\ar@{--}[rr]&&\cdot\ar[dr]\\&&&&\bullet\ar[ur]\ar[dr]\ar@{--}[rr]&&\cdot\\&\cdot\ar[dr]\ar@{--}[rr]&&\bullet\ar[ur]^{b}\ar@{--}[rr]&&\cdot\ar[ur]\\\cdot\ar[ur]\ar[dr]\ar@{--}[rr]&&\bullet\ar[dr]_{c}\ar[ur]^{a}\\&\cdot\ar[ur]\ar@{--}[rr]&&\bullet\ar[dr]_{d}\\&&&&\vdots} \xymatrix@C=1em@R=1em{\\\\\\\\\\\\\\&\cdot\ar[dr]\ar@{--}[rr]&&\bullet\ar[dr]^{b}\ar@{--}[rr]&&\cdot\ar[dr]\\\cdot\ar[ur]\ar[dr]\ar@{--}[rr]&&\bullet\ar[ur]^{a}\ar[dr]_{c}&&\bullet\ar[ur]\ar@{--}[rr]\ar[dr]&&\cdot\\&\cdot\ar[ur]\ar@{--}[rr]&&\bullet\ar[ur]_{d}\ar@{--}[rr]&&\cdot\ar[ur]}$$

We view $\Gamma_{R}$ and $\Gamma_{A}$ as bound quiver $K$-categories, bounded by all mesh-relations. This means that we identify $\Gamma_{R}$ with $K(\Gamma_{R}^{\op})=K(\Gamma_{R})^{\op}$ and $\Gamma_{A}$ with $K(\Gamma_{A}^{\op})=K(\Gamma_{A})^{\op}$. The group $G=\mathbb{Z}$ acts on $\Gamma_{R}$ by vertical translation, that is, all meshes lying on each of two vertical lines (on both sides of the line $L$) belong to the same orbit. We conclude that $\Gamma_{R}\slash G\cong\Gamma_{A}$. Denote the Galois covering $\Gamma_{R}\ra\Gamma_{A}$ by $\widetilde{F}$ and observe that $\widetilde{F}$ is equivalent with $(F_{\lambda}^{\Gamma})^{\op}$. Hence $\widetilde{F}^{\op}\cong F_{\lambda}^{\Gamma}$ and Theorem \ref{t11} yields $$(\widetilde{F}_{\lambda})^{\op}\cong(\widetilde{F}^{\op})_\lambda\cong (F_{\lambda}^{\Gamma})_{\lambda}\cong\Phi_F.$$ We show that the functor $\widetilde{F}_{\lambda}$ is not dense. Denote by $M(L)$ the linear $\Gamma_{R}$-module associated with the line $L$ and by $M(D)$ the $\Gamma_{A}$-module having $K$ in the vertices and identities on the arrows such that $\supp(M(D))=D$. Then $\widetilde{F}_{\bullet}(M(D))=M(L)$ and thus $\widetilde{F}_{\bullet}(M(D))$ is an indecomposable module of infinite dimension. This implies that the module $M(D)$ is of the second kind, because otherwise $\widetilde{F}_{\bullet}(M(D))$ would be of the form $\bigoplus_{g\in G}{}^{g}L$, for some indecomposable $L\in\mod(R)$. 

The above arguments show that the functor $\Phi_{F}:\CF(R)\ra\CF(A)$ is not dense. We emphasize that more detailed calculations given in \cite[Example 5.7]{P6'} allow to construct a one parameter family of functors of the second kind in the functor category $\CF(A)$. \epv
\end{exa}

\begin{exa}\label{e2} Assume that $R=\underline{KQ_{R}}\slash I_{R}$ where $Q_{R}$ is the following quiver: $$\xymatrix@C=0.8em@R=1em{&&&&1\ar[dr]^{\alpha}\ar[dddlll]_{\delta}&&&&&&1\ar[dr]^{\alpha}\ar[dddlll]_{\delta}&&&&&&1\ar[dr]^{\alpha}\ar[dddlll]_{\delta}&\\&&&&& 2\ar[dr]^{\beta}&&&&&& 2\ar[dr]^{\beta} &&&&&& 2\ar[dr]^{\beta}&&&\\&&&&&& 3\ar[dr]^{\gamma} &&&&&& 3\ar[dr]^{\gamma} &&&&&& 3\ar[dr]^{\gamma}&&\\\cdots &4 &&&&&& 4 &&&&&& 4 &&&&&& 4& \cdots}$$ and $I_{R}$ is generated by all zero relations of the form $\alpha\beta=\beta\gamma=0$. The group $G=\mathbb{Z}$ acts on $R$ by the horizontal translation and hence $R\slash G$ is the bound-quiver $K$-algebra $A=KQ_{A}\slash I_{A}$ where $Q_{A}$ is the quiver: $$\xymatrix@C=1em@R=1em{1\ar[dd]_{\delta}\ar[rr]^{\alpha}&& 2\ar[dd]^{\beta}\\&&\\3&&4\ar[ll]^{\gamma}}$$ and $I_{A}=\langle\alpha\beta,\beta\gamma\rangle$. We obtain a Galois covering functor $F:R\ra A$ with the covering group $G=\mathbb{Z}$ where $R$ is locally representation-finite simply connected and thus $A$ is standard. The Auslander-Reiten quivers $\Gamma_{R}$ and $\Gamma_{A}\cong\Gamma_{R}\slash G$ have the following shapes, respectively: $$\xymatrix@C=1em@R=0.95em{\\ \\ \\ \\  \dots\\} \xymatrix@C=1em@R=1em{&\cdot\ar[dr]\ar@{--}[rr]&&\cdot\ar[dr] &&&& \cdot\ar[dr]\ar@{--}[rr]&&\cdot\ar[dr]\\
 \cdot\ar[ur]\ar[dr]\ar@{--}[rr]&&\cdot\ar[ur]\ar[dr]\ar@{--}[rr]&&\cdot &&  \cdot\ar[ur]\ar[dr]\ar@{--}[rr]&&\cdot\ar[ur]\ar[dr]\ar@{--}[rr]&&\cdot\\&\cdot\ar[ur]\ar[dr]\ar@{--}[rr]&&\cdot\ar[ur]\ar[dr] &&&& \cdot\ar[ur]\ar[dr]\ar@{--}[rr]&&\cdot\ar[ur]\ar[dr]& && & \\
 \cdot\ar[ur]\ar@{--}[rr]&&\cdot\ar[ur]\ar@{--}[rr]&&\cdot\ar[dr]_{\mu}\ar@{--}[rr]&&\cdot\ar[ur]\ar@{--}[rr]&&\cdot\ar[ur]\ar@{--}[rr]&&\cdot\ar[dr]_{\mu}\ar@{--}[rr]&&\cdot \dots\\ &&&&&\cdot\ar[ur]_{\nu}&& &&&& \cdot \ar[ur]_{\nu}&}\xymatrix@C=1em@R=1em{&\cdot\ar[dr]\ar@{--}[rr]&&\cdot\ar[dr]\\\cdot\ar[ur]\ar[dr]\ar@{--}[rr]&&\cdot\ar[ur]\ar[dr]\ar@{--}[rr]&&\cdot\\&\cdot\ar[ur]\ar[dr]\ar@{--}[rr]&&\cdot\ar[ur]\ar[dr]\\\cdot\ar[ur]\ar@{--}[rr]&&\cdot\ar[ur]\ar@{--}[rr]&&\cdot\ar@/^/[dll]^{\tiny{\mu}}\\& & \cdot\ar@/^/[ull]^{\tiny{\nu}} & &}$$ Observe that $B=K(\Gamma_{R}^{\op})$ is locally support-finite, because we always have $\nu\mu=0$ and so the support of any indecomposable finite dimensional $B$-module is contained in a full subcategory $C$ of $B$ of the form: $$
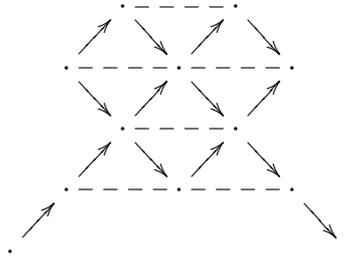
\xymatrix@C=1em@R=1em{&&\cdot\ar[dr]\ar@{--}[rr]&&\cdot\ar[dr]\\&\cdot\ar[ur]\ar[dr]\ar@{--}[rr]&&\cdot\ar[ur]\ar[dr]\ar@{--}[rr]&&\cdot\\&&\cdot\ar[ur]\ar[dr]\ar@{--}[rr]&&\cdot\ar[ur]\ar[dr]\\&\cdot\ar[ur]\ar@{--}[rr]&&\cdot\ar[ur]\ar@{--}[rr]&&\cdot\ar[dr]\\ \cdot\ar[ur]&&&&&&\cdot}$$ Denoting the Galois covering $\Gamma_{R}\ra\Gamma_{A}$ by $\widetilde{F}$, we conclude from \cite{DoSk,DoLeSk} that the functor $\widetilde{F}_{\lambda}$ is dense. As in the previous example, we get that the functor $\Phi_{F}:\CF(R)\ra\CF(A)$ is also dense. 
 
The standard knitting procedure of the Auslander-Reiten quiver shows that $B=K(\Gamma_{R}^{\op})$ is even locally representation-finite. It is worth to note that this fact follows directly from beautiful classical results of \cite{IT,IPTZ} where the authors give sufficient criteria for Auslander algebras to be representation finite. In \cite[Example 5.9]{P6'} we show a Galois covering functor $F:R\ra A$ such that $\Phi_{F}:\CF(R)\ra\CF(A)$ is dense but the category $K(\Gamma_{R}^{\op})$ is not locally representation-finite. \epv
\end{exa}

We refer the reader to \cite[Remark 5.6]{P6'} where we discuss possible generalizations of Theorem \ref{t10} to the representation-infinite case, at least to some extent. These could be used to showing examples of functors $\Phi_{F}:\CF(R)\ra\CF(A)$ which are not dense (similarly as in Example \ref{e1}), but is not sufficient for showing the density. An interesting open problem is to give a criterion for the density of the functor $\Phi_{F}:\CF(R)\ra\CF(A)$, in the spirit of \cite{DoSk,DoLeSk}. This boils down to appropriate generalization of the definition of a locally support-finite locally bounded $K$-category. We believe that the following property is worth studying in this context: \begin{center}\emph{For any module $X\in\ind(R)$, the union of supports of all indecomposable functors $T\in\CF(R)$ such that $T(X)\neq 0$ is contained in the support of a hom-functor.}\end{center} We recall that the support of $T\in\CF(R)$ is defined as the class of all modules $X\in\ind(R)$ such that $T(X)\neq 0$.

\section{Applications to Krull-Gabriel dimension}
In the final section of the paper we show applications of covering techniques described above in the theory of Krull-Gabriel dimension. We start with introductory section where we recall all necessary facts. In particular, we formulate the main results of \cite{P4,P6,P6'} which show relations between $\KG(R)$ and $\KG(A)$ where $R\ra A$ is a Galois covering. Then we outline applications to special classes of algebras. We base the section mostly on \cite{P4,J-PP1} and \cite[Section 6]{P6'}.

\subsection{Basic facts and notions}
We recall the definition of the Krull-Gabriel dimension of a locally bounded $K$-category. The general definition is similar, see for example Section 4 of \cite{P4}. Assume that $R$ is a locally bounded $K$-category. The associated \textit{Krull-Gabriel filtration} $(\CF(R)_{\alpha})_{\alpha}$ \cite{Po} is the filtration $$\CF(R)_{-1}\subseteq\CF(R)_{0}\subseteq\CF(A)_{R}\subseteq\hdots\subseteq\CF(R)_{\alpha}\subseteq\CF(R)_{\alpha+1}\subseteq\hdots$$ of $\CF(R)$ by Serre subcategories, defined recursively as follows: 
\begin{enumerate}[\rm(1)]
	\item $\CF(R)_{-1}=0$,
	\item $\CF(R)_{\alpha+1}$ is the Serre subcategory of $\CF(R)$ formed by all functors having finite length in the quotient category $\CF(R)\slash\CF(R)_{\alpha}$, for any ordinal number $\alpha$,
	\item $\CF(R)_{\beta}=\bigcup_{\alpha<\beta}\CF(R)_{\alpha}$, for any limit ordinal $\beta$.
\end{enumerate} Following \cite{Ge2,Geigle1985}, the \textit{Krull-Gabriel dimension} $\KG(R)$ of $R$ is defined as the smallest ordinal number $\alpha$ such that $\CF(R)_{\alpha}=\CF(R)$, if such a number exists. We set $\KG(R)=\infty$ if this is not the case. If $\KG(R)\in\NN$, then the Krull-Gabriel dimension of $R$ is \textit{finite}. If $\KG(R)=\infty$, then the Krull-Gabriel dimension of $R$ is \textit{undefined}.

Our main motivation to study Krull-Gabriel dimension comes from the following conjecture due to M. Prest \cite{Pr2}.

\begin{con}\label{00c1} Assume that $K$ is an algebraically closed field. A finite dimensional $K$-algebra $A$ is of domestic representation type if and only if the Krull-Gabriel dimension $\KG(A)$ of $A$ is finite.
\end{con} 

Observe that this conjecture fits naturally into the \emph{functorial approach} to representation theory of algebras, a line of research initiated by M. Auslander in \cite{Au,Au0}. Indeed, in \cite[Corollary 3.14]{Au} Auslander proved that representation-finite algebras are exactly those algebras for which all finitely presented functors are of finite length. In other words, if $A$ is an algebra, then $\KG(A)=0$ if and only if $A$ is of finite type. We refer to the Introduction of \cite{P6'} for the most up to date list of results supporting this conjecture, see also \cite[Section 1]{P4}. Clearly there are no results disproving the conjecture. 

We also recall that Prest made a similar conjecture, relating the representation type of an algebra $A$ with existence of \emph{super-decomposable pure-injective $A$-modules}, see \cite{Zi,Pr,Pr3}. Results of this paper can be applied to the second conjecture as well, but we do not discuss this matter here.

Assume that $\CR\subseteq\mod(A)$ is a class of $A$-modules and $N\in\mod(A)$. A homomorphism $\alpha_{N}\colon M_{N}\ra N$, where $M_{N}\in\CR$, is
a \textit{right $\CR$-approximation} of $N$ if and only if for any $L\in\CR$ and $a\colon L\ra N$ there is $b\colon L\ra M_{N}$ such that
$\alpha_{N}b=a$, that is, the following diagram $$\xymatrix{&N\\ L\ar[ur]^{a}\ar[r]^{b}&M_{N}\ar[u]_{\alpha_{N}}}$$ is commutative. We say
that $\CR$ is \textit{contravariantly finite} if and only if any module $N\in\mod(A)$ has a right $\CR$-approximation. If $\CR\subseteq\mod(A)$ is a contravariantly finite class of $A$-modules and $\CS$ is the smallest full subcategory of $\mod(A)$ closed under isomorphisms and direct summands such that $\CR\subseteq \ob(\CS)$, then $\CS$ is a contravariantly finite subcategory of $\mod(A)$ in the classical sense of \cite{AuRe}.

A full subcategory $B$ of a locally bounded $K$-category $R$ is \textit{convex} if and only if for any $n\geq 1$ and objects $x,z_{1},\dots,z_{n},y$ of $R$ the following condition is satisfied: if $x,y$ are objects of $B$ and the vector spaces of morphisms $$R(x,z_{1}), R(z_{1},z_{2}),\dots,R(z_{n-1},z_{n}),R(z_{n},y)$$ are nonzero, then $z_{1},\dots,z_{n}$ are objects of $B$. If $R=\und{KQ}/I$ is a path $K$-category of a bound quiver $(Q,I)$ and $B$ is a convex subcategory of $R$, then $B=\und{KQ'}/I'$, for some convex subquiver $Q'$ of the quiver $Q$. Recall that a full subquiver $Q'$ of $Q$ is \textit{convex} if and only if for any path $c_{1}\dots c_{n}$ from the vertex $x$ to the vertex $y$ in $Q$ such that $x,y\in Q'_{0}$ we have $s(c_{i})\in Q'_{0}$, for $i=1,\dots,n-1$.

In our results we frequently apply the following two straightforward facts. The first lemma is proved for example in \cite[Lemma 2.6]{P6'}, see also \cite[Section 4]{P4}. 

\begin{lm}\label{0l2} Assume that $R$ is a locally bounded $K$-category. 
\begin{enumerate}[\rm(1)]
  \item If $B$ is a convex subcategory of $R$, then the category $\mod(B)$ is a contravariantly finite subcategory of $\mod(R)$.
  \item If $B$ is a factor category of $R$, then the category $\mod(B)$ is a contravariantly finite subcategory of $\mod(R)$.
\end{enumerate} In both cases, $\KG(B)\leq\KG(R)$.
\end{lm}

The second lemma is based on \cite[Appendix B]{Kr}. We use it only for categories of finitely presented functors. 

\begin{lm}\label{0l1} Assume that $\CC,\CD$ are abelian categories and $F:\CC\ra\CD$ is an exact functor.
\begin{enumerate}[\rm(1)]
  \item If $F$ is faithful, then $\KG(\CC)\leq\KG(\CD)$.
  \item If $F$ is full and dense, then $\KG(\CD)\leq\KG(\CC)$.
\end{enumerate}
\end{lm}  

The following theorem is one of the main results of \cite{P6'}. In a slogan form, the theorem states that \emph{Galois $G$-coverings do not increase the Krull-Gabriel dimension}.

\begin{thm}\label{t3} Assume that $F:R\ra A$ is a Galois $G$-covering. Then $\KG(R)\leq\KG(A)$. In particular:
\begin{enumerate}[\rm(1)]
  \item If $\KG(R)$ is undefined, then $\KG(A)$ is undefined.
  \item If $\KG(A)$ is finite, then $\KG(R)$ is finite.
\end{enumerate}
\end{thm}

\begin{proof} It follows from Theorem \ref{t2} that $\Phi:\CF(R)\ra\CF(A)$ is exact. Observe that $\Phi$ is also faithful. Indeed, recall from \ref{t4} that $\bigoplus_{g\in G}{}_{\CR}(gT,T')\cong{}_{\CA}(\Phi(T),\Phi(T'))$ and this isomorphism is given by $(\iota_{g})_{g\in G}\mapsto\sum_{g\in G}\Phi(\iota_{g})$. Hence, by composing this isomorphism with the inclusion monomorphism ${}_{\CR}(T,T')\hookrightarrow\bigoplus_{g\in G}{}_{\CR}(gT,T')$, we obtain a monomorphism which gets the form $\iota\mapsto\Phi(\iota)$. This shows the faithfulness of the functor $\Phi$. Thus the inequality $\KG(R)\leq\KG(A)$ follows from \ref{0l1}.
\end{proof}

The above theorem establishes deep connections between covering theory, functor categories and Krull-Gabriel dimension. First connections of this kind are shown in \cite{P4,P6} where we give sufficient conditions for a Galois covering to preserve the Krull-Gabriel dimension. We recall some notions in order to formulate the result.

Assume that $G$ is an admissible group of $K$-linear automorphisms of $R$. A finite convex subcategory $B$ of the category $R$ is called \textit{fundamental domain} \cite{P6} if and only if for any $M\in\ind(R)$ there exists $g\in G$ such that $\supp({}^{g}M)\subseteq B$. We say that a locally bounded $K$-category $R$ is \textit{intervally finite} \cite[2.1]{BoGa} if and only if a convex hull of any finite full subcategory of $R$ is finite.

\begin{thm}\label{t8} Assume that $R$ is a locally support-finite and intervally finite locally bounded $K$-category. Assume that $G$ is an admissible torsion-free group of $K$-linear automorphisms of $R$ and let $F:R\ra A\cong R\slash G$ be the Galois $G$-covering. Then there exists a fundamental domain $B$ of $R$ and we have $\KG(R)=\KG(B)=\KG(A)$.
\end{thm} Results of \cite{P4} can be seen as showing only that $\KG(R)\leq\KG(A)$ if $R$ is locally support-finite and the group $G$ is torsion-free. We show in \cite{P6} that in this situation the opposite inequality also holds. The assumptions that we made are crucial for proofs in \cite{P6}, so it seems that in general $\KG(A)$ may be larger then $\KG(R)$. Examples of such case are not known.

\subsection{Standard self-injective algebras}

In this section we characterize the Krull-Gabriel dimension of locally support-finite repetitive $K$-categories and standard self-injective algebras. We base mainly on \cite{P4} and Theorem \ref{t8}. The following theorem is a part of Theorem (B) proved in \cite{AsSk4}.

\begin{thm}\label{s1} Assume that $K$ is an algebraically closed field and $A$ is a finite dimensional basic and connected $K$-algebra. The following conditions are equivalent.
\begin{enumerate}[\rm(1)]
	\item The repetitive $K$-category $\widehat{A}$ is locally support-finite and tame.
	\item There exists an algebra $B$ such that $\widehat{A}\cong\widehat{B}$ and $B$ is tilted of Dynkin type, tilted of Euclidean type or tubular.
\end{enumerate}
\end{thm}

Assume that $A$ is an algebra. A \textit{cycle} in $\ind(A)$ is a sequence $$M_{0}\stackrel{f_{1}}{\ra} M_{1}\ra\hdots\ra M_{r-1}\stackrel{f_{r}}{\ra}M_{r}=M_{0}$$ of nonzero nonisomorphisms in $\ind(A)$. This cycle is \textit{finite} if and only if $f_{1},\dots,f_{r}\notin\rad_{A}^{\infty}$. Following \cite{AsSk1} we call the algebra $A$ \textit{cycle-finite} if and only if all cycles in $\ind(A)$ are finite. The following result follows from \cite[Theorem 1.2, Corollary 1.5]{Sk5}.

\begin{thm}\label{s2} Assume that $K$ is an algebraically closed field and $A$ is a finite dimensional basic and connected $K$-algebra. If $A$ is cycle-finite, then $A$ is domestic if and only if $\KG(A)=2$.
\end{thm}

The following theorem describes the Krull-Gabriel dimension of locally support-finite $K$-categories. We base on \cite[Theorem 7.3]{P4}, but the proof is slightly changed in order to fit to our setting. Recall that it follows from \cite{DoSk2} that a locally bounded $K$-category $R$ is tame if and only if any finite full subcategory of $R$ is tame. Moreover, $R$ is wild if and only if there exists a finite full subcategory of $R$ which is wild.

\begin{thm}\label{s3} Assume that $K$ is an algebraically closed field and $A$ is a finite dimensional basic and connected $K$-algebra such that $\widehat{A}$ is locally support-finite. Then $\KG(\widehat{A})\in\{0,2,\infty\}$ and the following assertions hold. 
\begin{enumerate}[\rm(1)]
	\item $\KG(\widehat{A})=0$ if and only if $\widehat{A}\cong\widehat{B}$ where $B$ is some tilted algebra of Dynkin type.
	\item $\KG(\widehat{A})=2$ if and only if $\widehat{A}\cong\widehat{B}$ where $B$ is some tilted algebra of Euclidean type.
	\item $\KG(\widehat{A})=\infty$ if and only if $\widehat{A}$ is wild or $\widehat{A}\cong\widehat{B}$ where $B$ is some tubular algebra.
\end{enumerate} 
\end{thm}

\begin{proof} We show all assertions simultaneously. In any case, we consider Galois coverings of the form $\widehat{B}\ra B\slash\langle\nu_{B}\rangle$. Recall that the orbit category $B\slash\langle\nu_{B}\rangle$ is the \emph{trivial extension} $T(B)$ of $B$. Theorem \ref{s1} implies that the repetitive category $\widehat{B}$ is always locally support-finite, see also \cite{Sk0}. Hence we conclude from \ref{t8} that there is a finite convex subcategory $C$ of $\widehat{B}$ such that $\KG(C)=\KG(\widehat{B})$. Since $C$ is finite, we may also view $C$ as an algebra.

Assume that $B$ is a tilted algebra of Dynkin type. We show that $\KG(\widehat{B})=0$. Indeed, in this case $\widehat{B}$ is locally representation-finite \cite{AHR}. Thus $C$ is representation-finite, so $\KG(C)=0$ by Auslander's result. We get $\KG(\widehat{B})=\KG(C)=0$.

Assume that $B$ is a tilted algebra of Euclidean type. We show that $\KG(\widehat{B})=2$. It follows from \cite{AsSk1,Sk0,ANSk} that the tilted algebras of Euclidean type and their repetitive categories are cycle-finite of domestic type. This implies that the algebra $C$ is cycle-finite of domestic type and thus $\KG(C)=2$ by Theorem \ref{s2}. Hence $\KG(\widehat{B})=\KG(C)=2$.

Assume that $B$ is a tubular algebra. It follows from \cite{Ge2} that $\KG(B)=\infty$. Thus we get $\infty=\KG(B)\leq\KG(\widehat{B})$ by \ref{0l1}, because $B$ is a convex subcategory of $\widehat{B}$. Finally, if $\widehat{A}$ is wild, then there is a finite convex subcategory $B$ of $\widehat{A}$ such that $B$ is wild, see for example \cite{DoSk2}. Consequently, $\KG(B)=\infty$ and so $\KG(\widehat{A})=\infty$ as in the previous case. To complete the proof, observe that Theorem \ref{s1} yields $\KG(\widehat{A})\in\{0,2,\infty\}$. Hence the above implications are in fact equivalences.
\end{proof}

It follows from \cite{Sk0}, \cite{Sk4} that representation-infinite standard selfinjective algebras of polynomial growth are orbit algebras of the form $\widehat{B}\slash G$ where $B$ is a tilted algebra of Euclidean type or a tubular algebra and $G$ is an infinite cyclic admissible group of $K$-linear automorphisms of $\widehat{B}$. 

Therefore, in order to determine the Krull-Gabriel dimension of these algebras, it suffices to apply Theorem 7.3 and Theorem 6.3. 

The following theorem determines the Krull-Gabriel dimension of standard selfinjective algebras of polynomial growth. This theorem supports Conjecture 1.1 of M. Prest on the finiteness of Krull-Gabriel dimension, see Section 1.

\begin{thm}\label{s4} Assume that $A$ is a standard selfinjective algebra over an algebraically closed field $K$. Then the following assertions hold.
\begin{enumerate}[\rm (1)]
	\item If the algebra $A$ is representation-infinite domestic, then $\KG(A)=2$.
	\item If the algebra $A$ is nondomestic of polynomial growth, then $\KG(A)=\infty$.
\end{enumerate}
\end{thm}

\begin{proof} It follows from \cite{Sk0,Sk4} that representation-infinite standard self-injective algebras of polynomial growth are orbit algebras $\widehat{B}\slash G$ where $B$ is a tilted algebra of Euclidean type (then $\widehat{B}\slash G$ is representation-infinite domestic) or a tubular algebra (then $\widehat{B}\slash G$ nondomestic of polynomial growth) and $G$ is an infinite cyclic admissible group of $K$-linear automorphisms of $\widehat{B}$. Hence the assertions follows from Theorems \ref{t8} and \ref{s3}.
\end{proof}

It is worth to note that in \cite{Bu} the author gives the description of the \textit{Ziegler spectrum} (the space of isomorphism types of all indecomposable pure-injective modules \cite{Zi}) of the same class of self-injective algebras. His results imply \ref{s4} as well, but the methods used in \cite{Bu} are significantly different from ours.

\subsection{Cluster repetitive categories and cluster-tilted algebras}

This section is devoted to describe the Krull-Gabriel dimension of cluster repetitive categories and cluster-tilted algebras, see \ref{e44} for definitions and also \cite{As}. We base on the results from \cite{J-PP1} which in particular apply Theorem \ref{s3}.

Assume that $A,B$ are algebras and $\varphi\colon \mod(A)\ra\mod(B)$ is a $K$-linear additive covariant functor. Let $\CA=\mod(A)$ and $\CB=\mod(B)$. In general, the pull-up functor $\varphi_{\bullet}:\MOD(\CB)\ra\MOD(\CA)$ along $\varphi$ does not restrict to categories of finitely presented functors, that is, $\varphi_{\bullet}|_{\CF(B)}:\CF(B)\ra\mod(\CA)$ but $\Im(\varphi_{\bullet}|_{\CF(B)})\subsetneq\CF(A)$. This happens already for pull-ups along push-downs $F_{\lambda}$. Indeed, assume that $F:R\ra A$ is a Galois $G$-covering and consider the functor $\Psi=(F_{\lambda})_{\bullet}:\MOD(\CA)\ra\MOD(\CR)$. If $U\in\CF(A)$, then $$\Psi(U)(X)=U(F_{\lambda}(X))\cong U(F_{\lambda}({}^{g}X))=\Psi(U)({}^{g}X),$$ for any indecomposable $R$-module $X$ and $g\in G$. If $G$ is infinite and $\Psi(U)(X)\neq 0$, we conclude that $\Psi(U)({}^{g}X)\neq 0$ for infinite number of $g\in G$. It is easy to see that this cannot happen in case $\Psi(U)$ is a quotient of a hom-functor, because $G$ acts freely on the objects of $R$, see \cite[Lemma 5.1]{P4} for details.

An important example of a pull-up functor whose image lies in the category of finitely presented functors is considered in \cite[Theorem 1.3]{P6}. Indeed, it allows to show that $\KG(A)\leq\KG(R)$ holds under assumptions of Theorem \ref{t8}. This functor is also dense. We say that $\varphi\colon \mod(A)\ra\mod(B)$ is \textit{admissible} if and only if $\varphi$ is dense and $\Im(\varphi_{\bullet}|_{\CF(B)})\subseteq\CF(A)$,
that is, $U\circ\varphi$ is finitely presented, for any $U\in\CF(A)$. Admissible functors are particularly useful in the study of Krull-Gabriel dimension, because if a functor $\varphi\colon \mod(A)\ra\mod(B)$ is admissible, then $\KG(B)\leq\KG(A)$, see \cite[Proposition 2.1]{J-PP1}

The following fact is proved in \cite[Theorem 2.4]{J-PP1} and shows another important examples of admissible functors, this time related with contravariantly finite classes of modules. We apply it in the sequel.

\begin{thm} \label{s5}
Assume that $\varphi\colon\mod(A)\ra\mod(B)$ is a $K$-linear covariant full and dense functor. Assume that there is a
contravariantly finite class of modules $\CR_{\varphi}\subseteq\mod(A)$ such that $\Ker(\varphi)$ equals the class of all homomorphisms in $\mod(A)$
which factorize through $\CR_{\varphi}$. Then $\varphi\colon\mod(A)\ra\mod(B)$ is admissible and thus $\KG(B)\leq\KG(A)$.
\end{thm}

\begin{proof} Observe that if $U\in\CF(B)$ and ${}_{B}(-,X)\xrightarrow{_{B}(-,f)}{}_{B}(-,Y)\ra U\ra 0$ is exact, then we get the exact sequence $${}_{B}(\varphi(-),X)\xrightarrow{_{B}(\varphi(-),f)}{}_{B}(\varphi(-),Y)\ra U\circ\varphi\ra 0.$$ Hence it is enough to show that a functor ${}_{B}(\varphi(-),Z)\colon \mod(A)\ra\mod(K)$ belongs to $\CF(A)$, for any $Z\in\mod(B)$, because the category $\CF(A)$ is abelian. Since the functor $\varphi\colon \mod(A)\ra\mod(B)$ is dense, it is sufficient to show that ${}_{B}(\varphi(-),\varphi(N))\in\CF(A)$, for any fixed $N\in\mod(A)$. 

It is easy to see that $\wt{\varphi}=(\wt{\varphi}_{X})_{X\in\mod(A)}$ where $\wt{\varphi}_{X}\colon {}_{A}(X,N)\ra{}_{B}(\varphi(X),\varphi(N))$ is given by
the formula $\wt{\varphi}_{X}(f)=\varphi(f)$ is a natural transformation
${}_{A}(-,N)\ra{}_{B}(\varphi(-),\varphi(N))$. Assume that $\alpha_{N}\colon M_{N}\ra N$ is a right $\CR_{\varphi}$ approximation of the module $N$, for some $M_{N}\in\CR_{\varphi}$. Then the sequence
$${}_{A}(-,M_{N})\xrightarrow{_{A}(-,\alpha_{N})}{}_{A}(-,N)\stackrel{\wt{\varphi}}{\longrightarrow}{}_{B}(\varphi(-),\varphi(N))\ra 0$$
is exact, see the proof of \cite[Theorem 2.4]{J-PP1} for details. This implies that the functor $\varphi\colon \mod(A)\ra\mod(B)$ is admissible $\KG(B)\leq\KG(A)$.
\end{proof}

Assume that $\CR\subseteq\mod(A)$ is some class of $A$-modules. If $T\in\CF(A)$, then the class $\supp_{\CR}(T)=\{X\in\CR\mid T(X)\neq 0\}$ is called
the \textit{$\CR$-support} of $T$. We shall call the class $\CR$ \textit{hom-support finite} (in short, \textit{hs-finite}) if and only if the
$\CR$-support of a hom-functor ${}_{A}(-,N)$ is finite, for any $N\in\mod(A)$. Denote by $\add(\CR)$ the class of all finite direct sums of modules from the class $\CR$. Is is proved in \cite[Lemma 2.5]{J-PP1} that $\add(\CR)$ is contravariantly finite if $\CR\subseteq\mod(A)$ is a hs-finite class of $A$-modules. This is well-known for finite subcategories of $\mod(A)$, see for example \cite[Proposition 4.2]{AuSm}. 

The following theorem summarizes some of the main results of \cite{J-PP1}. 

\begin{thm} \label{s6} Assume that $C$ is a tilted algebra and $\widehat{C},\check{C},\wt{C}$ are the associated repetitive category, cluster repetitive category
            and cluster-tilted algebra, respectively. The following assertions hold.
\begin{enumerate}[\rm(1)]
	\item There exists a functor $\phi\colon \mod(\widehat{C})\ra\mod(\check{C})$ which is admissible.
	\item We have $\KG(\wt{C})=\KG(\check{C})\leq\KG(\widehat{C})$.
\end{enumerate}
\end{thm}

\begin{proof} $(1)$ Denote by $\CK_{C}$ the set $$\{\widehat{P}_{x},\tau^{1-i}\Omega^{-i}(C)\mid x\in(\widehat{C})_{0},i\in\ZZ\}$$ of modules from
$\mod(\widehat{C})$ where $\widehat{P}_{x}$ is an indecomposable projective $\widehat{C}$-module at the vertex $x\in(\widehat{C})_{0}$, $\tau=\tau_{\widehat{C}}$ the Auslander-Reiten translation in $\mod(\widehat{C})$ and $\Omega$ the syzygy functor. It is shown in \cite[Lemma 8, Theorem 9]{ABS} that there is full and dense $K$-linear functor $\phi:\mod(\widehat{C})\ra\mod(\check{C})$ such that $\Ker(\phi)$ equals the class of all homomorphisms in $\mod(\widehat{C})$ which factorize through $\add(\CK_{C})$. We prove in \cite[Proposition 3.3]{J-PP1} a technical fact that $\CK_{C}$ is hs-finite, so the functor $\phi\colon\mod(\widehat{C})\ra\mod(\check{C})$ is admissible by \ref{s5}.

$(2)$ We show in \cite[Theorem 3.1]{J-PP1} that there exists a fundamental domain $B$ of $\check{C}$ and hence $\KG(\check{C})=\KG(\widetilde{C})$ by \ref{s2}. By $(1)$ and \ref{s5} we conclude that $\KG(\check{C})\leq\KG(\widehat{C})$.
\end{proof}
 
In a sense, the following theorem refines Theorem \ref{s3}. It it the main result of \cite{J-PP1}, see Theorem 5.6.

\begin{thm} \label{s7} Assume that $K$ is an algebraically closed field, $C$ is a tilted $K$-algebra and $\widehat{C},\check{C},\wt{C}$ are the associated repetitive category, cluster repetitive category and cluster-tilted algebra, respectively. Then $\KG(\wt{C})=\KG(\check{C})=\KG(\widehat{C})\in\{0,2,\infty\}$ and the following assertions hold.
\begin{enumerate}[\rm(1)]
	\item $C$ is tilted of Dynkin type if and only if $\KG(\wt{C})=0$.
	\item $C$ is tilted of Euclidean type if and only if $\KG(\wt{C})=2$.
	\item $C$ is tilted of wild type if and only if $\KG(\wt{C})=\infty$.
\end{enumerate} In particular, a cluster-tilted algebra $\wt{C}$ has finite Krull-Gabriel dimension if and only if $\wt{C}$ is of domestic representation type.
\end{thm}

\begin{proof} We apply freely Theorem \ref{s3} and the fact that $\KG(\wt{C})=\KG(\check{C})\leq\KG(\widehat{C})$. We prove all
assertions simultaneously.

If $C$ is of Dynkin type, then $\KG(\wt{C})=\KG(\check{C})\leq\KG(\widehat{C})=0$, so $\KG(\wt{C})=\KG(\check{C})=\KG(\widehat{C})=0$. If $C$ is of Euclidean type, then $\KG(\wt{C})=\KG(\check{C})\leq\KG(\widehat{C})=2$, but $\KG(\wt{C})\neq 0$ since $\wt{C}$ is of infinite
representation type \cite{BMR} and $\KG(\wt{C})\neq 1$ by \cite{Kr2}. This yields $\KG(\wt{C})=\KG(\check{C})=\KG(\widehat{C})=2$. If $C$ is of wild type, then $\wt{C}$ is also of wild type by \cite{BMR} and hence we obtain $\KG(\wt{C})=\KG(\check{C})=\KG(\widehat{C})=\infty$.

Since the algebra $C$ is a tilted algebra either of Dynkin, or of Euclidean or of wild type, we conclude that the above implications can be replaced by
equivalences. This also shows that $\KG(\wt{C})=\KG(\check{C})=\KG(\widehat{C})\in\{0,2,\infty\}$. Moreover, if $C$ is tilted of Dynkin or
Euclidean type, then $\wt{C}$ is of domestic representation typ. This implies that the
class of cluster-tilted algebras supports the conjecture of Prest \ref{00c1}.
\end{proof}

\subsection{Algebras with strongly simply connected Galois-coverings}

In this section we are interested in a special class of algebras, namely, the tame algebras having strongly simply connected Galois coverings \cite{Sk3}. In this situation we consider Galois $G$-coverings $F:R\ra A$ for which the group $G$ is not torsion-free nor the push-down functor $F_{\lambda}$ is dense. Hence we are not allowed to apply Theorem \ref{t8} but only Theorem \ref{t3}. Recall that the class we consider is rather wide and contains, in particular, all special biserial algebras \cite[5.2]{DoSk}. Moreover, its properties suggest a new approach to the conjecture of Prest, see \cite[Section 7]{P6'} for a discussion of these matters. Here we base on Section 6 of \cite{P6'}.

We start with recalling some notions and facts needed to formulate Theorem \ref{t13} which is our main result. However, we view this reminder valuable since it contains fundamental concepts of covering theory. 

Assume that $R$ is a locally bounded $K$-category and $G$ is a group of $K$-linear automorphisms of $R$ acting freely on the objects of $R$. Given an $R$-module $M$ we denote by $G_{M}$ the \emph{stabilizer} $\{g\in G\mid {}^{g}M\cong M\}$ of $M$. An indecomposable $R$-module $M$ in $\Mod(R)$ is \textit{weakly $G$-periodic} \cite[2.3]{DoSk} if and only if $\supp (M)$ is infinite and $(\supp M)\slash G_{M}$ is finite. 

Assume that $D$ is a full subcategory of $R$ and $g\in G$. Then $gD$ denotes the full subcategory of $R$ formed by all objects $gx$, $x\in D$. The set $\{g\in G\mid gD=D\}$ is the \textit{stabilizer} $G_{D}$ of $D$. A \textit{line} in $R$ is a convex subcategory of $R$ which is isomorphic to the path category of a linear quiver, i.e. a quiver of type $\AA_{n}$, $\AA_{\infty}$ or ${}_{\infty}\AA_{\infty}$. A line $L$ is \textit{$G$-periodic} if and only if the stabilizer $G_{L}$ is nontrivial. Observe that in this case the quiver of the line $L$ is of the type ${}_{\infty}\AA_{\infty}$. It is well known that a $G$-periodic line $L$ in $R$ induces a \textit{band} in $R/G$, and every band induces a 1-parameter family of indecomposable $R/G$- modules, see for example \cite{DoSk}.

Assume that $L$ is a $G$-periodic line in $R$ {and let $Q_L$ be the full subquiver of $Q$, whose vertices are the objects of $L$}. Then the \textit{canonical} weakly $G$-periodic $R$-module is the module $M_{L}$ such that $M_{L}(x)=K$ for $x\in Q_{L}$, $M_{L}(x)=0$ for $x\notin Q_{L}$ and $M_{L}(\alpha)=id_{K}$ for any arrow $\alpha$ in $Q_{L}$. Note that $G_{M_{L}}\cong G_{L}\cong\ZZ$. If $M$ is a weakly $G$-periodic $R$-module and $M\cong M_{L}$ for some $G$-periodic line $L$ in $R$, then $M$ is called \textit{linear}.

Let us mention that the renowned result of \cite[Theorem 3.6]{DoSk} due to Dowbor and Skowro\'nski is widely used in representation theory. It's too technical to be included in this review but we recall its special version. This version is applied in \cite{Sk3} to obtain main results for algebras with strongly simply connected Galois coverings \cite[Theorem 2.4]{Sk3}.

Denote by $\CL_0$ a fixed set of representatives of all $G$-orbits of $G$-periodic lines in $R$. As usual, $K[T,T^{-1}]$ is the algebra of \emph{Laurent polynomials}. 

\begin{thm}\label{t15} Assume that $R$ is a locally bounded $K$-category over algebraically closed field $K$, $G$ an admissible group of $K$-linear automorphisms of $R$ and $F:R\ra A\cong R\slash G$ the Galois covering. Assume that the group $G$ acts freely on the isomorphism classes in $\ind(R)$ and every weakly $G$-periodic $R$-module is linear. The following assertions hold:
\begin{enumerate}[\rm(1)]
  \item Any module $Z\in\ind(A)$ of the second kind is of the form $V\otimes_{K[T,T^{-1}]}F_{\lambda}(M_{L})$ for some $L\in\CL_{0}$ and indecomposable finite dimensional $K[T,T^{-1}]$-module $V$.
  \item We have $$\Gamma_{A}=(\Gamma_{R}\slash G)\vee (\bigvee_{L\in\CL_{0}\snull}\Gamma_{K[T,T^{-1}]})$$ where $\Gamma_{K[T,T^{-1}]}$ denotes the Auslader-Reiten quiver of the category of finite dimensional $K[T,T^{-1}]$-modules.
\end{enumerate}
\end{thm} We recall that if $L\in\CL_0$, then $G_{M_{L}}=G_{L}\cong\ZZ$. Hence the group algebra $KG_{L}$ is isomorphic with $K[T,T^{-1}]$ and the canonical action of $G_{L}$ on $L$ gives a left $K[T,T^{-1}]$-module structure on the module $F_{\lambda}(M_{L})$. We refer to \cite{DoSk} for the details. It is worth to note that if $F:R\ra A$ is a universal Galois $G$-covering of a \emph{string algebra} $A$, then $R$ and $G$ satisfy conditions of the above theorem. This in particular means that modules of the first kind are exactly the \emph{string modules} whereas modules of the second kind are exactly the \emph{band modules}. 

We give more details about strongly simply connected algebras and $K$-categories. Assume that $A=kQ\slash I$ is a triangular algebra. Then $A$ is \textit{strongly simply connected} \cite{Sk2} if and only if the first Hochschild cohomology group $H^{1}(C,C)$ vanishes, for any convex subcategory $C$ of $A$. The classes of \textit{hypercritical} and \textit{pg-critical} algebras play a prominent role in the representation theory of strongly simply connected algebras. We refer to \cite{NoSk} for details on these classes. Here we only mention that hypercritical algebras are strictly wild and are classified by quivers and relations in \cite{Un}. In turn, pg-critical algebras are tame of non-polynomial growth \cite[Proposition 2.4]{Sk1} and are classified by quivers and relations in \cite{NoSk}. It follows by \cite[Theorem 7.1]{KaPa2} that hypercritical algebras and pg-critical algebras both have Krull-Gabriel dimension undefined, see also \cite[10.3]{Pr} and \cite{P5,GP} for wild algebras in general.

Assume that $R=\und{kQ}\slash I$ is a triangular locally bounded $k$-category. Then $R$ is \textit{strongly simply connected} if and only if the following two conditions are satisfied: 
\begin{itemize}
  \item $R$ is intervally finite,
  \item every finite convex subcategory of $R$ is strongly simply connected.
\end{itemize}

Tame algebras having strongly simply connected Galois coverings are studied in depth by A. Skowro\'nski in \cite{Sk3}. We only mention the following main result of this paper which we apply in the sequel. 

\begin{thm}\label{t12}\textnormal{(\cite[Theorem 2.6]{Sk3})} Assume that $R$ is a strongly simply connected locally bounded $K$-category, $G$ an admissible group of $K$-linear automorphisms of $R$ and let $A=R\slash G$. Then the following assertions hold:
\begin{enumerate}[\rm(1)]
  \item The algebra $A$ is of polynomial growth if and only if the category $R$ does not contain a convex subcategory which is hypercritical or pg-critical, and the number of $G$-orbits of $G$-periodic lines in $R$ is finite.
  \item The algebra $A$ is of domestic type if and only if the category $R$ does not contain a convex subcategory which is hypercritical, pg-critical or tubular, and the number of $G$-orbits of $G$-periodic lines in $R$ is finite.
\end{enumerate}
\end{thm}

By applying Theorem \ref{t3} we obtain the following fact.

\begin{thm}\label{t13} Assume that $R$ is a strongly simply connected locally bounded $K$-category, $G$ an admissible group of $K$-linear automorphisms of $R$ and $A=R\slash G$. If $\KG(A)$ is finite, then $A$ is of domestic type.
\end{thm}

\begin{proof} Assume that $\KG(A)$ is finite. Then Theorem \ref{t3} implies that $\KG(R)$ is finite and we show that in this case $R$ does not contain a convex subcategory which is hypercritical, pg-critical or tubular, and the number of $G$-orbits of $G$-periodic lines in $R$ is finite. First recall that if $B$ is a convex subcategory of $R$, then $\KG(B)\leq\KG(R)$ by Lemma \ref{0l2}. Therefore, if $R$ contains a convex subcategory $B$ which is hypercritical or pg-critical, then $\infty=\KG(B)\leq\KG(R)$ and thus $\KG(R)=\infty$. We conclude the same if $R$ contains a convex subcategory which is tubular, because tubular algebras have Krull-Gabriel dimension undefined \cite{Geigle1985}. Finally, assume that the number of $G$-orbits of $G$-periodic lines in $R$ is infinite. Without loss of generality we assume that $R$ does not contain a convex subcategory which is hypercritical or pg-critical. Then it follows from \cite[Lemma 6.1 (3), Theorem 6.2]{KaPa3} and their proofs that $A=R\slash G$ has a factor algebra $C$ which is string of non-domestic type. Hence $\KG(C)=\infty$ by \cite{Sch2} and since $\KG(C)\leq\KG(A)$ by Lemma \ref{0l2} we get $\KG(A)=\infty$, contrary to our assumption that $\KG(A)$ is finite. These arguments show that we can apply Theorem \ref{t12} (2) and conclude that $A$ is of domestic type. 
\end{proof}

We believe that the converse of this theorem is valid, so Prest's conjecture holds for the class of algebras with strongly simply connected Galois coverings. We refer to \cite[Remark 6.4]{P6'} for discussion of this problem. 

\bibsection

\end{document}